\documentclass[11pt]{article}
\usepackage{graphics}
\usepackage[dvips]{graphicx}

\usepackage{amsmath,amssymb,amsthm}
\usepackage[utf8]{inputenc}
\usepackage{tikz}
\usepackage{subfigure}
\usepackage{color}
\usepackage{comment}
\usepackage{float}
\usepackage[a4paper,margin=2.5cm]{geometry}

\title{
Simultaneous Convergent Continued Fraction Algorithm for Real and $p$-adic Fields with Applications to Quadratic Fields
}

\author{Shin-ichi Yasutomi}

\newtheorem{thm}{Theorem}[section]
\newtheorem{cor}[thm]{Corollary}
\newtheorem{lem}[thm]{Lemma}

\newtheorem{conj}[thm]{Conjecture}

\theoremstyle{definition}
\newtheorem{defn}{Definition}[section]
\newtheorem{example}{Example}[section]

\theoremstyle{remark}
\newtheorem{rem}{Remark}[section]

\numberwithin{equation}{section}

\usepackage{color}

\begin{document}

\maketitle


\begin{abstract}
Let $p$ be a prime number and $K$ be a field with embeddings into $\mathbb{R}$ and $\mathbb{Q}_p$.
We propose an algorithm that generates continued fraction expansions converging in $\mathbb{Q}_p$ and is expected to simultaneously converge in both $\mathbb{R}$ and $\mathbb{Q}_p$.
This algorithm produces finite continued fraction expansions for rational numbers.
In the case of $p=2$ and if $K$ is a quadratic field, the continued fraction expansions generated by this algorithm converge in $\mathbb{R}$, and they are eventually periodic or finite.
For an element $\alpha$ in $K$, let $p_n/q_n$ denote the $n$-th convergent.
  There exist constants  $u_1$ and $u_2$ in ${\mathbb R}_{>0}$ with $u_1 + u_2 = 2$, and constants $C_1$ and $C_2$ in ${\mathbb R}_{>0}$
  such that $|\alpha - p_n/q_n| < C_1/|q_n|^{u_1}$ and $|\alpha - p_n/q_n|_2 < C_2/|q_n|^{u_2}$.
Here, $|\cdot|_2$ represents the $2$-adic distance.
For prime numbers $p > 2$, we present numerical experiences.
\end{abstract}

\section{Introduction}
The continued fraction algorithm not only provides the best approximation in Diophantine approximation of real numbers, but also possesses various favorable properties and holds significant positions in various mathematical domains.
Not only the regular continued fraction algorithm, but also continued fraction algorithms with various features have been proposed. Research on these algorithms remains active up to the present day.
Let $p$ be a prime number and $\mathbb{Q}_p$ be
the completion of  $\mathbb{Q}$ with respect to  the $p$-adic topology.
For $u\in \mathbb{Q}_p$, let $v_p(u)$ be the valuation of $u$ and $|v|_p:=\frac{1}{p^{v_p(u)}}$. 
A continued fraction expansion algorithm in $\mathbb{Q}_p$ has yet to be discovered 
that rivals the regular continued fraction expansion algorithm in the real numbers.
Mahler\cite{M} initiated the first attempt at $p$-adic continued fractions.
Schneider \cite{S} and Ruban \cite{Ru} independently proposed different algorithms during the same period, both contributing significantly to the field of continued fraction expansion algorithms for $\mathbb{Q}_p$
(see for example \cite{Ro}).
The continued fraction expansion algorithms by Schneider and Ruban are known to reveal that rational numbers do not necessarily possess finite expansions (see \cite{Bu}, \cite{L}, \cite{W}) and that quadratic irrationals do not always possess periodic expansions (see \cite{CVZ}, \cite{O}, \cite{We}).
The algorithm we propose in this paper follows in the lineage of Ruban's continued fraction expansion algorithm. Let us introduce necessary notation.
Let $J$ be a representative system modulo $p$.
It is well known that every $u\in {\mathbb Q}_p$ can be written as
\begin{align*}
u=\sum_{n\in \mathbb{Z}}c_np^n,  \ \  c_n\in J,
\end{align*}
where $c_k=0$ for $k<v_p(u)$.
We define
\begin{align*}
\lfloor u \rfloor_p^{J}:=\sum_{n\in \mathbb{Z}_{\leq 0}}c_np^n,\quad
\lceil u \rceil_p^{J}:=\sum_{n\in \mathbb{Z}_{< 0}}c_np^n.
\end{align*}
For the standard representative $J=\{0,1,\ldots,p-1\}$, we denote $\lfloor \cdot \rfloor_p^{J}$ and $\lceil \cdot \rceil_p^{J}$ by $\lfloor \cdot \rfloor_p$ and $\lceil \cdot \rceil_p$ respectively.

Ruban's continued fraction algorithm is applied to $\alpha\in {\mathbb Q}_p$ as outlined below.
Starting with $\alpha_0=\alpha$, 
we define sequences $\{a_n\}$ and $\{\alpha_n\}$ as follows:
\begin{align*}
a_n=\lfloor \alpha_n \rfloor_p,\quad
\alpha_{n+1}=\dfrac{1}{\alpha_{n}-a_n}.
\end{align*}
Browkin \cite{Br} defined a following algorithm similar to Ruban's by considering $J=\{-\frac{p-1}{2},\ldots,\frac{p-1}{2}\}$ when $p$ is an odd prime, demonstrating finite continued fraction expansions for rational numbers.
Starting with $\alpha_0=\alpha$, 
we define sequences $\{a_n\}$ and $\{\alpha_n\}$ as follows:
\begin{align}\label{Browkin0}
a_n=\lfloor \alpha_n \rfloor_p^J,\quad
\alpha_{n+1}=\dfrac{1}{\alpha_{n}-a_n}.
\end{align}
In Browkin \cite{Br2}, the following continued fraction algorithm is provided, and the defining expression varies depending on whether $n$ is even or odd.
Let $\alpha\in {\mathbb Q}_p$ and $J=\{-\frac{p-1}{2},\ldots,\frac{p-1}{2}\}$. 
Starting with $\alpha_0=\alpha$, 
The sequences $\{a_n\}$ and $\{\alpha_n\}$ are defined as follows:
\begin{align}\label{Browkin}
&a_n=\begin{cases}
\lfloor \alpha_n \rfloor_p^J,&\text{if $n$ is even,}\\
\lceil \alpha_n \rceil_p^{J},&\text{if $n$ is odd and $v_p(\alpha_n-\lceil \alpha_n \rceil_p^{J})=0$,}\\
\lceil \alpha_n \rceil_p^{J}-sign(\lceil \alpha_n \rceil_p^{J}),&\text{if $n$ is odd and $v_p(\alpha_n-\lceil \alpha_n \rceil_p^{J})\ne 0$,}
\end{cases}\\
&\alpha_{n+1}=\dfrac{1}{\alpha_{n}-a_n}.\nonumber
\end{align}
Browkin \cite{Br2} showed the experimental results about this algorithm.
Barbero, Cerruti and Murru \cite{BCM} demonstrated that rational numbers have finite continued fraction expansions.
Murru and Romeo \cite{MR} introduced the following modified algorithm,
 which is based on Algorithm (\ref{Browkin}). They demonstrated that this algorithm improves upon Browkin's algorithms in several aspects.
Let $\alpha\in {\mathbb Q}_p$ and $J=\{-\frac{p-1}{2},\ldots,\frac{p-1}{2}\}$. 
Starting with $\alpha_0=\alpha$, 
The sequences $\{a_n\}$ and $\{\alpha_n\}$ are defined as follows:
\begin{align}\label{MReq}
&a_n=\begin{cases}
\lfloor \alpha_n \rfloor_p^J,&\text{if $n$ is even,}\\
\lceil \alpha_n \rceil_p^{J},&\text{if $n$ is odd,}
\end{cases}\\
&\alpha_{n+1}=\dfrac{1}{\alpha_{n}-a_n}.\nonumber
\end{align}
Murru, Romeo, and Santilli \cite{MRS} defined a class including the algorithm mentioned in (\ref{MReq}),
 and discussed convergence properties as well as the finite expansion of rational numbers.

Let $K$ be a field that has an embedding into $\mathbb{R}$ and $\mathbb{Q}_p$ respectively.
Assume $\sigma_{\infty}$ gives an embedding into $\mathbb{R}$ and $\sigma_p$ gives an embedding into $\mathbb{Q}_p$.
We define a novel algorithm inspired by the approaches in (\ref{Browkin}) and (\ref{MReq}), aiming to achieve simultaneous rational approximations in both $\mathbb{R}$ and $\mathbb{Q}_p$ for elements of $K$. While the process involves selecting integers that are close to given numbers, it is designed to maintain proximity within the respective topologies of $\mathbb{R}$ and $\mathbb{Q}_p$.
Let $\alpha \in K$. We denote $\sigma_{\infty}(\alpha)$ by $\alpha_{\infty}$ (or alternatively, $(\alpha)_{\infty})$   and $\sigma_{p}(\alpha)$ by $\alpha_{\langle p \rangle}$ (or alternatively, $(\alpha)_{\langle p \rangle})$.
For a vector $\mathbf{a}=(x_1,\ldots,x_n)\in K^n$, we denote
 $((x_1)_{\infty},\ldots,(x_n)_{\infty})$ by $\mathbf{a}_{\infty}$(or alternatively, $(\mathbf{a})_{\infty}$) and
 $((x_1)_{\langle p \rangle},\ldots,(x_n)_{\langle p \rangle})$ by $\mathbf{a}_{\langle p \rangle}$(or alternatively, $(\mathbf{a})_{\langle p \rangle}$).

\begin{defn}\label{transformation}
We define a transformation  $F_{p,0}$  by
for $\alpha\in K$
\begin{align*}
F_{p,0}(\alpha):=\dfrac{1}{\alpha-mp-\left\lfloor\alpha_{\langle p \rangle} \right\rfloor_p},
\end{align*}
where among $i \in \mathbb{Z}$ for which $|\alpha_{\infty}-ip-\left\lfloor\alpha_{\langle p \rangle} \right\rfloor_p|$ attains its minimum, let $m$ be the smallest value.
We define $b^{(0)}(\alpha):=mp+\left\lfloor \alpha_{\langle p \rangle} \right\rfloor_p$.
We define a transformation $F_{p,1}$ by
for $\alpha\in K$
\begin{align*}
F_{p,1}(\alpha):=\dfrac{1}{\alpha-m-\left\lfloor \alpha_{\langle p \rangle} \right\rfloor_p},
\end{align*}
where among $i \in \mathbb{Z}$ for which $|\alpha_{\infty}-i-\left\lfloor \alpha_{\langle p \rangle} \right\rfloor_p|$ 
attains its minimum, let $m$ be the smallest value.
We define $b^{(1)}(\alpha):=m+\left\lfloor \alpha_{\langle p \rangle} \right\rfloor_p$.
\end{defn}

\begin{rem}
We remark that $F_{p,0}$ and $F_{p,1}$ are determined independently of the representative system modulo $p$.
\end{rem}

\begin{defn}\label{algol1}
We define an algorithm  as follows:
Let $\alpha=\alpha_0 \in K$. 
We define $\{\alpha_n\}$ and $\{b_n(\alpha)\}$recursively as:\\
\begin{align*}
&\alpha_{n+1}:=\begin{cases}
F_{p,0}(\alpha_{n})& \text{if $n\equiv 0 \mod\ 2$},\\
F_{p,1}(\alpha_{n})& \text{if $n\equiv 1 \mod\ 2$},\\
\end{cases}\\
&\text{and}\\
&b_{n}(\alpha):=\begin{cases}
b^{(0)}(\alpha_n)& \text{if $n\equiv 0 \mod\ 2$},\\
b^{(1)}(\alpha_n)& \text{if $n\equiv 1 \mod\ 2$}.
\end{cases}
\end{align*}
The algorithm halts if $F_{p,0}(\alpha_{n})$ or $F_{p,1}(\alpha_{n})$ is not defined,
 resulting in $\alpha_{n+1},\ldots$ not being defined.
\end{defn}

By this algorithm, for $\alpha \in K$, the sequence $\{b_n(\alpha)\}_{n\in \mathbb{Z}{\geq 0}}$ is generated, allowing us to consider its formal continued fraction expansion $\left[b_0(\alpha);b_1(\alpha),\ldots\right]$,
namely,
\begin{align*}
b_0(\alpha)+\cfrac{1}{b_1(\alpha) + \cfrac{1}{b_2(\alpha) + \cfrac{1}{b_3(\alpha)+\cfrac{1}{\ldots}}}}.
\end{align*}

We define convergents of continued fraction in the usual manner.

\begin{defn}\label{convergents1}
For $n\in \mathbb{Z}_{> 0}$, $p_n$ and $q_n$ are defined by
\begin{align*}
\begin{pmatrix}
p_{n}&p_{n-1}\\
q_{n}&q_{n-1}
\end{pmatrix}
=
\begin{pmatrix}
b_0(\alpha)&1\\
1&0
\end{pmatrix}
\begin{pmatrix}
b_1(\alpha)&1\\
1&0
\end{pmatrix}
\cdots
\begin{pmatrix}
b_n(\alpha)&1\\
1&0
\end{pmatrix},
\end{align*}
where we define $p_{-1}=1$ and $q_{-1}=0$.
\end{defn} 
We note that $p_n$ and $q_n$ are  in  $\mathbb{Z}[\frac{1}{p}]$ for 
all $n\in \mathbb{Z}_{\geq 0}$.

There are several studies that have dealt with continued fraction expansions that converge in both $\mathbb{Q}_p$ and $\mathbb{R}$ (e.g., \cite{ABCM}, \cite{BE}, \cite{P}).
The discussion on the periodicity of Ruban's continued fraction expansions is effectively utilized in \cite{CVZ} by embedding them into real numbers.
Bekki \cite{Bekki} proposed a continued fraction algorithm for complex quadratic irrationals, considering reductions in both the $\mathbb{Q}_p$ and $\mathbb{C}$ and demonstrated that these numbers have an eventually periodic continued fraction expansion.

It is not difficult to see the convergence of our algorithm in $\mathbb{Q}_p$.
We demonstrate that every rational number has a finite continued fraction expansion.
We show that for irrational numbers in $K$, the continued fraction expansion of each converges in $\mathbb{R}$ and eventually becomes periodic when $p=2$ and $K$ is a real quadratic field. This corresponds to Lagrange's theorem for regular continued fractions (see for example \cite{HW}).  
We also demonstrate that the accuracy of approximation is evaluated by utilizing the denominators of the convergents when $p=2$ and $K$ is a real quadratic field.
Ridout \cite{Ri} extended Roth's theorem to establish limits on simultaneous approximations of algebraic numbers, both in the field of real numbers and across several $p$-adic fields. Our algorithm provide the extent of approximation that approaches the boundary of those limits.
Bedocchi \cite{Be} showed a condition for purely periodic expansions with Algorithm (\ref{Browkin0}) for numbers with eventually periodic expansions.
for numbers with eventually periodic expansions.
 Similarly, Murru and Romeo \cite{MR} showed a condition for purely periodic expansions
  with respect to Algorithm (\ref{MReq}) for numbers possessing eventually periodic expansions.
  These results parallel Galois's theorem for regular continued fractions.
 We give a necessary condition for the continued fraction expansion of numbers in relation to our algorithm to become purely periodic.
In the determination of the periodic points in continued fraction expansions, the natural extension associated with the dynamical system related to the algorithm is commonly employed (see for example \cite{DK}).
Regarding our algorithm, we construct an analogue of the natural extension.             
Now, it is important to note that our algorithm is defined even when $K$ is embedded into $\mathbb{C}$. 
From numerical computation examples, it suggests that the algorithm works in such cases as well.
We remark that since there exist embeddings from $\mathbb{Q}_p$ to $\mathbb{C}$, our algorithm can be interpreted as an algorithm over $\mathbb{Q}_p$.
On the other hand, there are also approaches of deriving the continued fraction expansion using algebraic conditions of elements in $K$
 (see \cite{Br2}, \cite{STY}, \cite{STY2}).

\section{Fundamental properties}
In this chapter, we outline the fundamental properties of our algorithm.
\begin{lem}\label{lemord}
Let $\alpha\in K$. Then, for even $n\in \mathbb{Z}_{> 0}$, 
we have $v_p((\alpha_{n})_{\langle p \rangle})\leq 0$. For odd $n\in \mathbb{Z}_{> 0}$, it holds that $v_p((\alpha_{n})_{\langle p \rangle})< 0$.
For all $n\in \mathbb{Z}_{>0}$, $v_p((\alpha_{n})_{\langle p \rangle})=v_p(b_n(\alpha))$.
\end{lem}
\begin{proof}
Since it holds that for some integer $m$,
$\alpha-b^{(0)}(\alpha)=\alpha-mp-\left\lfloor \alpha_{\langle p \rangle} \right\rfloor_p$,
we have $v_p(\alpha_{\langle p \rangle}-b^{(0)}(\alpha))>0$.
Therefore, if $\alpha-b^{(0)}(\alpha)\ne 0$,  
we have $v_p((\alpha_{1})_{\langle p \rangle})=v_p(\frac{1}{\alpha_{\langle p \rangle}-b^{(0)}(\alpha)})<0$.
Since it holds that for some integer $m'$,
$\alpha_1-b^{(1)}(\alpha_1)=\alpha_1-m'-\left\lfloor (\alpha_1)_{\langle p \rangle} \right\rfloor_p$,
we have $v_p(\alpha_1-b^{(1)}(\alpha_1))\geq 0$.
Therefore, from the fact that $v_p((\alpha_1)_{\langle p \rangle})<0$, we have 
$v_p(b^{(1)}(\alpha_1))=v_p((\alpha_1)_{\langle p \rangle})$.
Similarly, the proof of the theorem can be established inductively.
\end{proof}

\begin{lem}\label{convergentsp}
For all $n\in \mathbb{Z}_{> 0}$, $\displaystyle v_p(q_n)=\sum_{k=1}^nv_p(b_k(\alpha))$.
\end{lem}
\begin{proof}
We will prove the claim by induction on $n$.
By the definition, we have $v_p(q_1)=v_p(b_1(\alpha))$.  
From the fact that $q_2=b_2(\alpha)b_1(\alpha)+1$ and Lemma \ref{lemord},  
we have $v_p(q_2)=v_p(b_1(\alpha))+v_p(b_2(\alpha))$.
Let $n\geq 2$. We assume that the clam holds for $n$ and $n-1$. 
From Lemma \ref{lemord}, we have $v_p(b_{n+1}(\alpha))+v_p(b_{n}(\alpha))<0$.
Therefore, we have $v_p(q_{n+1})=v_p(b_{n+1}(\alpha)q_{n}+q_{n-1})=v_p(b_{n+1}(\alpha)q_{n})=\sum_{k=1}^{n+1}v_p(b_k(\alpha))$.
\end{proof}

By Lemmas \ref{lemord} and \ref{convergentsp}, we observe that $q_n \ne 0$ for all $n \in \mathbb{Z}_{> 0}$, implying that $\frac{p_n}{q_n}$ are well-defined.

Following lemma  gives a sufficient condition for the convergence of continued fractions. 
\begin{lem}[\cite{MRS}]\label{MRS}
Let $b_0, b_1,\ldots \in \mathbb{Z}[\frac{1}{p}]$ be an infinite sequence such that
\begin{align*}
v_p(b_nb_{n+1})<0,
\end{align*}
for all $n>0$.
Then, the continued fraction $\left[b_0; b_1, \ldots\right]$ is convergent to a $p$-adic number.
\end{lem}

From Lemma \ref{lemord} and \ref{MRS}, we see  that $\left[b_0(\alpha);b_1(\alpha),\ldots\right]$
is convergent to a $p$-adic number. 
The equality $\alpha_{\langle p \rangle}=\left[b_0(\alpha); b_1(\alpha),\ldots\right]$ can be obtained later.

\begin{lem}\label{ineq1}
Let $\alpha\in K$.
Then, $\alpha_{\infty}-b^{(0)}(\alpha) \in \left(-\frac{p}{2},\frac{p}{2}\right]$ and $\alpha_{\infty}-b^{(1)}(\alpha)
\in \left(-\frac{1}{2},\frac{1}{2}\right]$.
\end{lem}

\begin{proof}
From Definition \ref{transformation} we see that
\begin{align*}
&|\alpha_{\infty}-b^{(0)}(\alpha)|=\min \{|\alpha_{\infty}-ip-\left\lfloor\alpha_{\langle p \rangle} \right\rfloor_p|\mid i\in \mathbb{Z}\},\\
&|\alpha_{\infty}-b^{(1)}(\alpha)|=\min \{|\alpha_{\infty}-i-\left\lfloor\alpha_{\langle p \rangle} \right\rfloor_p|\mid i\in \mathbb{Z}\},
\end{align*}
which imply that $|\alpha_{\infty}-b^{(0)}(\alpha)|\leq \frac{p}{2}$ and $|\alpha_{\infty}-b^{(1)}(\alpha)|\leq \frac{1}{2}$.
If it holds that $\alpha_{\infty}-b^{(0)}(\alpha)=-\frac{p}{2}$, 
since  it holds that $\alpha_{\infty}-(b^{(0)}(\alpha)-p)=\frac{p}{2}$, which contradicts
the definition of  $b^{(0)}(\alpha)$.
 Therefore, we have $\alpha_{\infty}-b^{(0)}(\alpha) \in \left(-\frac{p}{2},\frac{p}{2}\right]$.
 Similarly, we have  $\alpha_{\infty}-b^{(1)}(\alpha)
\in \left(-\frac{1}{2},\frac{1}{2}\right]$.
\end{proof}

We will present the following theorem using the same argument as in \cite{MR}.

\begin{thm}\label{rational}
Let $\alpha\in \mathbb{Q}$.
Then, $\{\alpha_n\}$ is a finite sequence.
\end{thm}
\begin{proof}
We assume that $\{\alpha_n\}$ is an infinite sequence.
Let $n\in \mathbb{Z}_{> 0}$.
From Lemma \ref{lemord}, we can set
\begin{align}
 &\alpha_{n}=\dfrac{N_{n}}{D_{n}p^{i_n}}, \ \ \text{with $(N_{n},D_{n})=1$ and $p \nmid N_{n}D_{n}$,}\label{ND}\\
&b_n(\alpha)=\dfrac{c_n}{p^{i_n}}, \ \ \text{with $p \nmid c_n$} \nonumber,
\end{align}
where $N_{n}$, $D_{n}$, and $c_n$ are integers.

 From Lemma \ref{lemord},  we have $i_{2n}\geq 0$, $i_{2n+1}> 0$, and $i_{2n+2}\geq 0$.  
From the fact that $\alpha_{k+1}=\frac{1}{\alpha_{k}-b_k(\alpha)}$ for $k=2n, 2n+1$, we have
\begin{align}
&N_{2n+1}(N_{2n}-c_{2n}D_{2n})=p^{i_{2n}+i_{2n+1}}D_{2n+1}D_{2n},\label{ND2}\\
&N_{2n+2}(N_{2n+1}-c_{2n+1}D_{2n+1})=p^{i_{2n+1}+i_{2n+2}}D_{2n+2}D_{2n+1}.\nonumber
\end{align}
From (\ref{ND}) and (\ref{ND2}), we have
\begin{align*}
&|N_{2n+1}|=|D_{2n}|,\\
&|N_{2n+2}|=|D_{2n+1}|,\\ 
&p^{i_{2n}+i_{2n+1}}|D_{2n+1}|=|N_{2n}-c_{2n}D_{2n}|,\\
&p^{i_{2n+1}+i_{2n+2}}|D_{2n+2}|=|N_{2n+1}-c_{2n+1}D_{2n+1}|.
\end{align*}
Therefore, using Lemma \ref{ineq1}, we have
\begin{align*}
&p^{i_{2n}+i_{2n+1}}|D_{2n+1}|=|N_{2n}-c_{2n}D_{2n}|\\
&=p^{i_{2n}}|D_{2n}|\left|\dfrac{N_{2n}}{D_{2n}p^{i_{2n}}}-\dfrac{c_{2n}}{p^{i_{2n}}}\right|\\
&\leq \dfrac{p^{i_{2n}+1}|D_{2n}|}{2},
\end{align*}
which implies 
\begin{align}\label{D2n+1}
|D_{2n+1}|\leq \dfrac{|D_{2n}|}{2p^{i_{2n+1}-1}}.
\end{align}
Similarly, we have
\begin{align*}
&p^{i_{2n+1}+i_{2n+2}}|D_{2n+2}|=|N_{2n+1}-c_{2n+1}D_{2n+1}|\\
&=p^{i_{2n+1}}|D_{2n+1}|\left|\dfrac{N_{2n+1}}{p^{i_{2n+1}}D_{2n+1}}-\dfrac{c_{2n+1}}{p^{i_{2n+1}}}\right|\\
&<p^{i_{2n+1}}|D_{2n+1}|,
\end{align*}
which implies 
\begin{align}\label{D2n+2}
|D_{2n+2}|<\dfrac{|D_{2n+1}|}{p^{i_{2n+2}}}.
\end{align}
From Lemma \ref{lemord}, we have $i_{2n+1}\geq 1$ and $i_{2n+2}\geq 0$.
Therefore, by inequalities (\ref{D2n+1}) and (\ref{D2n+2}), we obtain  inequalities
\begin{align*}
|D_{2n+1}|\leq \dfrac{|D_{2n}|}{2} \text{ and }|D_{2n+2}|< |D_{2n+1}|.
\end{align*}
Therefore, the sequence $\{|D_n|\}$ is strictly decreasing, which 
is a contradiction.
\end{proof}

\begin{example}
We give  the continued fraction expansion of $\dfrac{5}{13}$ for some prime numbers by the algorithm.
\begin{align*}
\begin{array}{ll}
\left[1;-\dfrac{13}{8}\right]&\text{ $p=2$},\\
&\\
\left[-1;\dfrac{2}{9},2\right]&\text{ $p=3$},\\
&\\
\left[0;\dfrac{13}{5}\right]&\text{ $p=5$},\\
&\\
\left[2;-\dfrac{2}{7},-3\right]&\text{ $p=7$}.
\end{array}
\end{align*}

\end{example}

By considering the continued fraction expansion of $\alpha-b_0(\alpha)$, we are able to confine the initial point to a limited range, thus allowing us to investigate the counterpart of the Gauss map in the regular continued fraction.

\begin{defn}\label{transformationT}
We define a map $T_{p,0}:K^{\times}\to K$ by
for $\alpha\in K^{\times}$
\begin{align*}
T_{p,0}(\alpha):=\dfrac{1}{\alpha}-mp-\left\lfloor \left(\dfrac{1}{\alpha}\right)_{\langle p \rangle} \right\rfloor_p,
\end{align*}
where among $i \in \mathbb{Z}$ for which $|\frac{1}{\alpha}-ip-\lfloor \left(\frac{1}{\alpha}\right)_{\langle p \rangle} \rfloor_p|$ attains its minimum, let $m$ be the smallest value.
We define $a^{(0)}(\alpha):=mp+\left\lfloor \frac{1}{\alpha} \right\rfloor_p$.
We define a map $T_{p,1}:K^{\times}\to K$ by
for $\alpha\in K^{\times}$
\begin{align*}
T_{p,1}(\alpha):=\dfrac{1}{\alpha}-m-\left\lfloor \left(\dfrac{1}{\alpha}\right)_{\langle p \rangle}  \right\rfloor_p,
\end{align*}
where among $i \in \mathbb{Z}$ for which $|\frac{1}{\alpha}-i-\left\lfloor (\frac{1}{\alpha})_{\langle p \rangle} \right\rfloor_p|$
attains its minimum, let $m$ be the smallest value.
We define $a^{(1)}(\alpha):=m+\left\lfloor (\frac{1}{\alpha})_{\langle p \rangle} \right\rfloor_p$.
\end{defn}

From Lemma \ref{ineq1}, we observe that $\alpha_{\infty}-b^{(0)}(\alpha) \in \left(-\frac{p}{2},\frac{p}{2}\right]$ and it is evident that $v_p(\alpha-b_0(\alpha))>0$.

\begin{defn}\label{property I}
We say that $\alpha$ has {\it property $I$}, if  it satisfies $\alpha\in \left(-\frac{p}{2},\frac{p}{2}\right]$ and $v_p(\alpha_{\langle p \rangle})>0$.
\end{defn}

Let $\alpha\in K/\mathbb{Q}$ with property $I$.

\begin{defn}\label{algol2}
We define an algorithm  as follows:
Let $\alpha=\alpha_{(1)} \in K$. 
We define $\{\alpha_{(n)}\}_{n\in \mathbb{Z}_{> 0}}$ and $\{a_n(\alpha)\}_{n\in \mathbb{Z}_{> 0}}$recursively as follows:\\
If $\alpha_{(n)}\ne 0$,
\begin{align*}
&\alpha_{(n+1)}:=\begin{cases}
T_{p,0}(\alpha_{(n)})& \text{if $n\equiv 0 \mod\ 2$},\\
T_{p,1}(\alpha_{(n)})& \text{if $n\equiv 1 \mod\ 2$},\\
\end{cases}\\
&\text{and}\\
&a_{n}(\alpha):=\begin{cases}
a^{(0)}(\alpha_n)& \text{if $n\equiv 0 \mod\ 2$},\\
a^{(1)}(\alpha_n)& \text{if $n\equiv 1 \mod\ 2$}.
\end{cases}
\end{align*}
If $\alpha_{n}=0$, terminate.
\end{defn}

Namely, we consider the following continued fraction expansion:
\begin{align*}
\cfrac{1}{a_1(\alpha) + \cfrac{1}{a_2(\alpha) + \cfrac{1}{a_3(\alpha)+\cfrac{1}{\ldots}}}}.
\end{align*}

We define convergents of continued fraction in the usual manner.
In the sequel, we assume that $\alpha\in K/\mathbb{Q}$ with property $I$.

\begin{defn}\label{convergents}
For $n\in \mathbb{Z}_{> 0}$, $p_n$ and $q_n$ are defined by
\begin{align*}
\begin{pmatrix}
p_{n-1}&p_{n}\\
q_{n-1}&q_{n}
\end{pmatrix}
=
\begin{pmatrix}
0&1\\
1&a_1(\alpha)
\end{pmatrix}
\begin{pmatrix}
0&1\\
1&a_2(\alpha)
\end{pmatrix}
\cdots
\begin{pmatrix}
0&1\\
1&a_n(\alpha)
\end{pmatrix},
\end{align*}
where we define $p_{0}=0$ and $q_{0}=1$.
\end{defn} 
We note that $p_n$ and $q_n$ are  in  $\mathbb{Z}[\frac{1}{p}]$ for 
all $n\in \mathbb{Z}_{\geq 0}$.

\begin{lem}\label{twoalgol}
Let $\alpha\in K$ with property $I$.
Then, $b^{(0)}(\alpha)=0$.
\end{lem}
\begin{proof}
From the fact that $v_p(\alpha_{\langle p \rangle})>0$, we have 
$\left\lfloor \alpha_{\langle p \rangle} \right\rfloor_p=0$.
Since $-\frac{p}{2}<\alpha_{\infty}\leq \frac{p}{2}$, 
we have 
$b^{(0)}(\alpha)=p\min\{i \mid |\alpha_{\infty}-ip|=\min\{|\alpha_{\infty}-jp|\mid j\in \mathbb{Z}\}\}=0$.

\end{proof}

\begin{cor}\label{rational2}
Let $\alpha\in \mathbb{Q}$ with property $I$.
Then, $\{\alpha_{(n)}\}$ is a finite sequence.
\end{cor}
\begin{proof}
From Theorem \ref{rational},  we see that $\{\alpha_{n}\}$ is a finite sequence, where
the sequence is generated by  Algorithm (Definition \ref{algol1}) to $\alpha$.
Let $\{\alpha_{n}\}=\{\alpha_{n}\}_{n=0}^k$.
From Lemma \ref{twoalgol}, we have $\alpha_1=\frac{1}{\alpha_{0}}=\frac{1}{\alpha_{(1)}}$.
Through induction, we establish the following for $n=1,\ldots, k$: $\alpha_k=\frac{1}{\alpha_{(k)}}$. Considering the fact that $\alpha_{k}-b_{k}(\alpha)=0$, we can conclude that $\{\alpha_{(n)}\}=\{(\alpha_{(n)})_{n=1}^k\}$.
\end{proof}

The following lemmas are evident from Definitions \ref{transformationT} and \ref{algol2}.
\begin{lem}\label{alphainK}
Let $\alpha\in K/\mathbb{Q}$ with property $I$.
Then, for $n\in \mathbb{Z}_{> 0}$, if  $n\equiv 0 \mod\ 2$, then $|(\alpha_{(n)})_{\infty}|<\frac{1}{2}$and
if  $n\equiv 1 \mod\ 2$, then $|(\alpha_{(n)})_{\infty}|<\dfrac{p}{2}$.
\end{lem}

\begin{lem}\label{lemord2}
Let $\alpha\in K$. Then, for even $n\in \mathbb{Z}_{> 0}$, 
we have $v_p((\alpha_{(n)})_{\langle p \rangle})\geq 0$. For odd $n\in \mathbb{Z}_{> 0}$, it holds that $v_p((\alpha_{(n)})_{\langle p \rangle})> 0$.
For all $n\in \mathbb{Z}_{>  0}$, $v_p((\alpha_{(n)})_{\langle p \rangle})=-v_p(a_n(\alpha))$.
\end{lem}
 
The following lemma  can be proven in a similar manner to Lemma \ref{convergentsp}.

\begin{lem}\label{convergentsp2adic}
For all $n\in \mathbb{Z}_{> 0}$, $\displaystyle v_p(q_n)=\sum_{k=1}^nv_p(a_k(\alpha))$ and
$\displaystyle v_p(p_n)=\sum_{k=2}^nv_p(a_k(\alpha))$.
\end{lem}

By Lemma \ref{convergentsp2adic}, it is established that for all $n\in \mathbb{Z}_{> 0}$, the convergents $\frac{p_n}{q_n}$ are well defined and are not equal to $0$.
Therefore, as it is well-known, for all $n\in \mathbb{Z}_{> 0}$, we have
\begin{align*}
[0;a_1(\alpha),\ldots, a_n(\alpha)]=\dfrac{p_n}{q_n}.
\end{align*}

We extend ${T}_{p,\epsilon}$ for $\epsilon\in \{0,1\}$ as a skew product on $K^2$ as follows:
\begin{defn}\label{product}
For $\epsilon\in \{0,1\}$, we define maps $\hat{T}_{p,\epsilon}:(K^{\times})^2\to K^2$  as
for $(\alpha,\beta)\in (K^{\times})^2$
\begin{align*}
\hat{T}_{p,\epsilon}(\alpha,\beta):=\left(T_{p,\epsilon}(\alpha),\dfrac{1}{\beta}-a^{(\epsilon)}(\alpha)\right).
\end{align*}
\end{defn}

Using $\hat{T}_{2,\epsilon}$ for $\epsilon\in \{0,1\}$, for $\mathbf{a}=(\alpha,\beta)\in K^2$, 
we define the following algorithm. 

\begin{defn}\label{algol3}
We define an algorithm  as follows:
Let $\mathbf{a}=(\alpha,\beta)=\mathbf{a}_{(1)} \in K^2$. 
We define $\{\mathbf{a}_{(n)}\}_{n\in \mathbb{Z}_{\geq 1}}$ recursively as:\\
If $\mathbf{a}_{(n)}\in (K^{\times})^2$,
\begin{align*}
&\mathbf{a}_{(n+1)}:=\begin{cases}
\hat{T}_{p,0}(\mathbf{a}_{(n)})& \text{if $n\equiv 0 \mod\ 2$},\\
\hat{T}_{p,1}(\mathbf{a}_{(n)})& \text{if $n\equiv 1 \mod\ 2$}.
\end{cases}
\end{align*}
If $\mathbf{a}_{(n)}\notin (K^{\times})^2$, terminate.
\end{defn}
We note that the first coordinate of  $\mathbf{a}_{(n)}$ is $\alpha_{(n)}$. 

We define $\tau:\mathbb{Z}\to \{0,1\}$ as follows: for $n\in \mathbb{Z}$, 
\begin{align*}
\tau(n):=\begin{cases} 0 & \text{if $n$ is even},\\
                      1 & \text{if $n$ is odd}.
\end{cases}                      
\end{align*}

\begin{defn}\label{domain2}
We define
\begin{align*}
&D_{p}^0:=\{(x,y)\in \mathbb{Q}_p^2\mid v_p(x)\geq 0, v_p(y)<0\},\\
&D_{p}^1:=\{(x,y)\in \mathbb{Q}_p^2\mid v_p(x)>0, v_p(y)\leq 0\}.
\end{align*}
\end{defn}

\begin{lem}\label{Letalpha=}
Let $(\alpha,\beta)\in (K/\mathbb{Q})\times K^{\times}$.
Then, 
\begin{enumerate}
\item[(1)] 
If $(\alpha_{\langle p \rangle},\beta_{\langle p \rangle})\in D_{p}^0$, then $(\hat{T}_{p,0}(\alpha,\beta))_{\langle p \rangle}\in D_{p}^1$. 
\item[(2)] 
If $(\alpha_{\langle p \rangle},\beta_{\langle p \rangle})\in D_{p}^1$, then $(\hat{T}_{p,1}(\alpha,\beta))_{\langle p \rangle}\in D_{p}^0$. 
\end{enumerate}
\end{lem}

\begin{proof}
(1) We assume that $(\alpha_{\langle p \rangle},\beta_{\langle p \rangle})\in D_{p}^0$.
Then, $\hat{T}_{p,0}(\alpha,\beta)=(\frac{1}{\alpha}-a^{(0)}(\alpha),\frac{1}{\beta}-a^{(0)}(\alpha))$.
By Definition \ref{transformationT},
we see that $v_p((\frac{1}{\alpha}-a^{(0)}(\alpha))_{\langle p \rangle})>0$.
Since $v_p(a^{(0)}(\alpha))=v_p((\frac{1}{\alpha})_{\langle p \rangle})\leq 0$ and $v_p((\frac{1}{\beta})_{\langle p \rangle})>0$,
we have $v_p((\frac{1}{\beta}-a^{(0)}(\alpha))_{\langle p \rangle})\leq 0$.
Thus, we have claim (1).\\
(2) We assume that $(\alpha_{\langle p \rangle},\beta_{\langle p \rangle})\in D_{p}^1$.
Then, $\hat{T}_{p,1}(\alpha,\beta)=(\frac{1}{\alpha}-a^{(1)}(\alpha),\frac{1}{\beta}-a^{(1)}(\alpha))$.
By Definition \ref{transformationT},
we see that $v_p((\frac{1}{\alpha}-a^{(1)}(\alpha))_{\langle p \rangle})\geq 0$.
Since $v_p((\frac{1}{\alpha})_{\langle p \rangle})<0$, we have $v_p(a^{(1)}(\alpha))<0$.
Therefore, since $v_p((\frac{1}{\beta})_{\langle p \rangle})\geq 0$, we conclude that  $v_p((\frac{1}{\beta}-a^{(1)}(\alpha))_{\langle p \rangle})=v_p((a^{(1)}(\alpha))_{\langle p \rangle})<0$.
Thus, we have claim (2).
\end{proof}

\begin{thm}\label{conver}
Let $\alpha\in K$ have property $I$.
Then, in $\mathbb{Q}_p$
\begin{align*}
\lim_{n\to \infty} \dfrac{p_n}{q_n}=\alpha_{\langle p \rangle}.
\end{align*}
\end{thm}
\begin{proof}
If $\alpha\in \mathbb{Q}$, then from Corollary \ref{rational2},
the claim of the theorem holds.
We assume that $\alpha\in K/\mathbb{Q}$.
Let $n\in \mathbb{Z}_{> 0}$.
Let $\mathbf{a}=(\alpha,\frac{p_n}{q_n})$.
We apply Algorithm (Definition \ref{algol3}) to $\mathbf{a}$.
By induction, we obtain the following:\\
For $1\leq j \leq n$,
\begin{align*}
\mathbf{a}_{(j)}=(\alpha_{(j)},[0;a_j(\alpha),a_{j+1}(\alpha),\ldots,a_{n}(\alpha)]).
\end{align*}
and
\begin{align*}
\mathbf{a}_{(n+1)}=(\alpha_{(n+1)},0).
\end{align*}
We put $(u_k.v_k)=(\mathbf{a}_{(k)})_{\langle p \rangle}$ for $1\leq k \leq n+1$.
For $k\in \mathbb{Z}_{> 0}$, we see that $\displaystyle u_{k+1}=\frac{1}{u_{k}}-a_k(\alpha)$ and $\displaystyle v_{k+1}=\frac{1}{v_{k}}-a_k(\alpha)$
, which implies
\begin{align*}
|u_k-v_k|_p=|u_kv_k|_p|u_{k+1}-v_{k+1}|_p.
\end{align*}
Therefore, for $n\in \mathbb{Z}_{> 0}$, we have
\begin{align}\label{u0v0u0v0}
|u_1-v_1|_p=|u_1|_p|v_1|_p\ldots |u_n|_p|v_n|_p|u_{n+1}-v_{n+1}|_p=|u_1|_p|v_1|_p\ldots |u_n|_p|v_n|_p |u_{n+1}|_p.
\end{align}
From Lemma \ref{lemord2} and \ref{convergentsp2adic}, we observe that $|u_k|_p \leq \frac{1}{p}$ and $|v_k|_p \leq \frac{1}{p}$ hold  for odd  $k$, and  $|u_k|_p \leq 1$ and $|v_k|_p \leq 1$ hold for even  $k$ for $1\leq k \leq n+1$.
Therefore, from the equality (\ref{u0v0u0v0}), we have
\begin{align*}
\left|\alpha_{\langle p \rangle}-\dfrac{p_n}{q_n}\right|_p\leq \dfrac{1}{p^n},
\end{align*}
 which implies $\lim_{n\to \infty} \dfrac{p_n}{q_n}=\alpha_{\langle p \rangle}$ in $\mathbb{Q}_p$.
\end{proof}

\begin{cor}\label{inftyp_nq_ncor}
Let $\alpha\in K$. For $n\in \mathbb{Z}_{\geq 0}$, let $\frac{p_n}{q_n}$ be the convergent
of $\alpha$ related to  Algorithm (Definition \ref{algol1}).
Then,  in $\mathbb{Q}_p$
\begin{align*}
\lim_{n\to \infty}\dfrac{p_n}{q_n}=\alpha_{\langle p \rangle}.
\end{align*}
\end{cor}
\begin{proof}
Since $\beta=\alpha-b_0(\alpha)$ has  property $I$, we can deduce that in $\mathbb{Q}_p$
\begin{align*}
\displaystyle\lim_{n\to \infty} \dfrac{p'_n}{q'_n}=\beta_{\langle p \rangle},
\end{align*}
where $\frac{p'_n}{q'_n}$ represents the $n$-th convergent of  $\beta$ obtained using Algorithm (Definition \ref{algol2}).
Considering that the $\frac{p_n}{q_n}=b_0(\alpha)+\frac{p'_n}{q'_n}$ for $n\in \mathbb{Z}_{> 0}$, 
we arrive at the theorem's conclusion.
\end{proof}

\begin{example}
Let $\alpha$ be the root of $x^2-\dfrac{17}{5}x+\dfrac{63}{25}$ with $\alpha_{\infty}=\dfrac{17+\sqrt{37}}{10}$ and
$\alpha_{\langle 7 \rangle}\equiv 2 \mod\ 7$.
We give  the continued fraction expansion of $\alpha$ for $p=7$ obtained using Algorithm (Definition \ref{algol1}) as follows:
\begin{align*}
\left[2; \dfrac{22}{7}, \overline{12, -\dfrac{2}{7}, -5, -\dfrac{29}{7}, -12, \dfrac{2}{7}, 5, \dfrac{29}{7}}\right].
\end{align*}
According to Corollary \ref{inftyp_nq_ncor}, this expansion converges to $\alpha_{\langle 7 \rangle}$.
By utilizing the expansion's periodicity, we can compute the convergents of $\alpha_{\infty}$. This enables us to demonstrate the convergence of this expansion towards $\alpha_{\infty}$.
\end{example}

\section{Case of $p=2$}
In this section, we consider the case of $p=2$. 
Let $\alpha\in K/\mathbb{Q}$ have property $I$.
In this chapter, we apply Algorithm (Definition \ref{algol2}) to $\alpha$.

\begin{defn}\label{domain}
We define
\begin{align*}
&I_{\infty}^0:=\{x\in \mathbb{R}\mid |x|\leq \frac{1}{2}\},\\
&I_{\infty}^1:=\{x\in \mathbb{R}\mid |x|\leq 1\},\\
&D_{\infty}^0:=\{(x,y)\in \mathbb{R}^2\mid |x|\leq \frac{1}{2},|x-y|> \frac{1}{2}\},\\
&D_{\infty}^1:=\{(x,y)\in \mathbb{R}^2\mid |x|\leq 1,|x-y|> 1\}.
\end{align*}
\end{defn}

\begin{lem}\label{Dinfty}
The following assertions hold:\\
Let $(u,v)\in (K^{\times})^2$. 
\begin{enumerate}
\item[(1)]
If $(u_{\infty},v_{\infty})\in D_{\infty}^0$, then $(\hat{T}_{2,0}(u,v))_{\infty}\in D_{\infty}^1.$
\item[(2)]
If $(u_{\infty},v_{\infty})\in D_{\infty}^1$, then $(\hat{T}_{2,1}(u,v))_{\infty}\in D_{\infty}^0.$
\end{enumerate}
\end{lem}
\begin{proof}
(1) We assume that $(u_{\infty},v_{\infty})\in D_{\infty}^0$.
We assume that $u_{\infty}>0$ without loss of generality.
Let $(u',v')=(\hat{T}_{2,0}(u,v))_{\infty}$.
By  Definition \ref{transformationT} we have
\begin{align*}
|u'|=\min\{\left|\frac{1}{u_{\infty}}-2i-\left\lfloor \left(\frac{1}{u}\right)_{\langle 2 \rangle} \right\rfloor_2\right|\mid i\in \mathbb{Z} \},
\end{align*}
which implies $|u'|\leq 1$. 
Since $(u',v')=(\frac{1}{u_{\infty}}-a^{(0)}(u),\frac{1}{v_{\infty}}-a^{(0)}(u))$,
we have 
\begin{align*}
|u'-v'|=\left|\frac{1}{u_{\infty}}-\frac{1}{v_{\infty}}\right|=\left|\frac{u_{\infty}-v_{\infty}}{u_{\infty}v_{\infty}}\right|.
\end{align*}
If it holds that $v_{\infty}<0$, we would have $|u'-v'|>\frac{1}{u_{\infty}}\geq 2$.
We assume $v_{\infty}>0$.
Since $|u_{\infty}-v_{\infty}|> \frac{1}{2}$ and $u_{\infty},v_{\infty}>0$, there exists $\epsilon>0$ such that
$v_{\infty}=u_{\infty}+\dfrac{1}{2}+\epsilon$.
We see 
\begin{align}\label{uv}
u_{\infty}v_{\infty}=u_{\infty}(u_{\infty}+\dfrac{1}{2}+\epsilon)\leq \dfrac{1}{2}(1+\epsilon)<\dfrac{1}{2}+\epsilon.
\end{align}
Therefore, we have 
\begin{align*}
|u'-v'|=\left|\frac{\frac{1}{2}+\epsilon}{u_{\infty}v_{\infty}}\right|>1.
\end{align*}
(2) We assume that $(u_{\infty},v_{\infty})\in D_{\infty}^1$.
We assume that $u_{\infty}>0$ without loss of generality.
Let $(u',v')=\hat{T}_{2,1}(u_{\infty},v_{\infty})$.
By  Definition \ref{transformationT} we have
\begin{align*}
|u'|=\min\{\left|\frac{1}{u_{\infty}}-i-\left\lfloor \left(\frac{1}{u}\right)_{\langle 2 \rangle} \right\rfloor_2\right|\mid i\in \mathbb{Z}\},
\end{align*}
which implies $|u'|\leq \frac{1}{2}$. 
Similarly, we have
\begin{align*}
|u'-v'|=\left|\frac{u_{\infty}-v_{\infty}}{u_{\infty}v_{\infty}}\right|.
\end{align*}
If it holds that $|v_{\infty}|\leq 2$, we would have $|u'-v'|=\frac{|u_{\infty}-v_{\infty}|}{|u_{\infty}v_{\infty}|}>\frac{1}{2}$.
We assume that $|v_{\infty}|>2$.
Then, we have $|u'-v'|=\dfrac{1}{u_{\infty}}-\dfrac{1}{v_{\infty}}>\dfrac{1}{u_{\infty}}-\dfrac{1}{2}\geq \dfrac{1}{2}$.
\end{proof}

\begin{rem}\label{infinity}
The lemma is extended as follows by adding $\infty$ as a domain.
We define
\begin{align*}
&\mathcal{D}_{\infty}^0:=\{(x,y)\in \mathbb{R}\times (\mathbb{R}\cup \{\infty\})\mid |x|\leq \frac{1}{2},|x-y|> \frac{1}{2}\},\\
&\mathcal{D}_{\infty}^1:=\{(x,y)\in \mathbb{R}\times (\mathbb{R}\cup \{\infty\})\mid |x|\leq 1,|x-y|> 1\}.
\end{align*}
The maps $\hat{T}_{2,\epsilon} (\epsilon\in \{0,1\})$ are extended from their domain $(K^{\times})^2$ to $K^{\times}\times (K^{\times}\cup {\infty})$ in the usual manner.
Then, we have:\\
For $(u,v)\in K^{\times}\times (K^{\times}\cup {\infty})$,  
\begin{enumerate}
\item[(1)]
If $(u_{\infty},v_{\infty})\in \mathcal{D}_{\infty}^0$, then $(\hat{T}_{2,0}(u,v))_{\infty}\in D_{\infty}^1,$
\item[(2)]
If $(u_{\infty},v_{\infty})\in \mathcal{D}_{\infty}^1$, then $(\hat{T}_{2,1}(u,v))_{\infty}\in D_{\infty}^0.$
\end{enumerate}

\end{rem}

\begin{lem}\label{Letalpha}
Let $\mathbf{a}=(\alpha,\beta)\in (K/\mathbb{Q})^2$ with  $\alpha\ne\beta$.
Then, either (1) or (2) holds.
\begin{enumerate}
\item[(1)] There exists $n\in \mathbb{Z}_{> 0}$ such that 
$(\mathbf{a}_{(2n)})_{\infty}\in D_{\infty}^0$,
which is equivalent to saying that there exists $n\in \mathbb{Z}_{> 0}$ such that $(\mathbf{a}_{(2n-1)})_{\infty}\in D_{\infty}^1$.

\item[(2)] For all $n\in \mathbb{Z}_{>0}$, $(\mathbf{a}_{(2n)})_{\infty}\notin D_{\infty}^0$
and one of the following holds:
\begin{enumerate}
\item[(i)] $\displaystyle \lim_{m\to \infty} (\mathbf{a}_{(2m)})_{\infty}=\left(\frac{1}{2},1\right)$ and
$\displaystyle \lim_{m\to \infty} (\mathbf{a}_{(2m+1)})_{\infty}=(-1,-2)$. 
\item[(ii)] 
 $\displaystyle \lim_{m\to \infty} (\mathbf{a}_{(2m)})_{\infty}=\left(-\frac{1}{2},-1\right)$ and
$\displaystyle \lim_{m\to \infty} (\mathbf{a}_{(2m+1)})_{\infty}=(1,2)$.
\end{enumerate}
\end{enumerate}
\end{lem}
\begin{proof}
For each $n\in \mathbb{Z}_{>0}$, we put $(u_n,v_n):=(\mathbf{a}_{(n)})_{\infty}$.
We assume that for all $n\in \mathbb{Z}_{>0}$, $(\mathbf{a}_{(2n)})_{\infty}\notin D_{\infty}^0$.
From Lemma \ref{Dinfty}, we see that for all $n\in \mathbb{Z}_{\geq 0}$, $(\mathbf{a}_{(2n+1)})_{\infty}\notin D_{\infty}^1$.
Therefore, we have
\begin{align}
&|u_n-v_n|\leq \frac{1}{2} & \text{if $n\equiv 0 \mod\ 2$},\label{unvne}\\
&|u_n-v_n|\leq 1 & \text{if $n\equiv 1 \mod\ 2$}.\label{unvno}
\end{align}
Since it holds that $|v_n|\leq |u_n|+|v_n-u_n|$  for each $n\in \mathbb{Z}_{> 0}$,
we have 
\begin{align}\label{vnleq1}
&|v_n|\leq  1, &\text{if $n\equiv 0 \mod\ 2$},\\
&|v_n|\leq  2, &\text{if $n\equiv 1 \mod\ 2$}.\label{vnleq2} 
\end{align}
For $n\in \mathbb{Z}_{> 0}$, we see that $\displaystyle u_{n+1}=\frac{1}{u_{n}}-a_n(\alpha)$ and $\displaystyle v_{n+1}=\frac{1}{v_{n}}-a_n(\alpha)$
, which implies
\begin{align*}
|u_n-v_n|=|u_nv_n||u_{n+1}-v_{n+1}|.
\end{align*}
Therefore, for $n\in \mathbb{Z}_{> 0}$, we have
\begin{align}\label{u0v0}
|u_1-v_1|=|u_1v_1|\ldots |u_nv_n||u_{n+1}-v_{n+1}|\leq |u_1v_1|\ldots |u_nv_n|.
\end{align}
From (\ref{vnleq1}), (\ref{vnleq2}), and Lemma \ref{alphainK}, we can derive the inequality
for any $m \in \mathbb{Z}_{> 0}$
\begin{align}\label{u2m}
|u_{2m}||v_{2m}||u_{2m+1}||v_{2m+1}| < 1.
\end{align}
We assume that
there exists $0<\epsilon<1$ such that for infinitely many $j\in \mathbb{Z}_{>0}$
$|u_{2j}|<\dfrac{1}{2}\epsilon$.
Since $|u_{2j}|<\dfrac{1}{2}\epsilon$ implies 
$|u_{2j}||v_{2j}||u_{2j+1}||v_{2j+1}|<\epsilon$,
from (\ref{u0v0}) and (\ref{u2m})
we have 
\begin{align*}
|u_1-v_1|\leq \lim_{n\to \infty} |u_1v_1|\ldots |u_nv_n|=0.
\end{align*}
However, it contradicts $\alpha\ne\beta$.
Therefore, we have
\begin{align}\label{limu2n}
\lim_{n\to \infty}|u_{2n}|=\frac{1}{2}.
\end{align}
Similarly, we have
\begin{align}\label{limu2n+1}
\lim_{n\to \infty}|v_{2n}|=1,
\lim_{n\to \infty}|u_{2n+1}|=1,
\text{and }
\lim_{n\to \infty}|v_{2n+1}|=2.
\end{align}
Therefore, for $\epsilon=\dfrac{1}{8}$, there exists $N\in \mathbb{Z}_{> 0}$
such that for $n\geq N$, the following inequalities hold:
\begin{align}
&\left||u_{2n}|-\frac{1}{2}\right|<\epsilon,\hspace{1cm}
\left||v_{2n}|-1\right|<\epsilon, \label{u2n}\\
\text{and}\nonumber\\
&\left||u_{2n+1}|-1\right|<\epsilon,\hspace{1cm}
\left||v_{2n+1}|-2\right|<\epsilon.\label{u2n1}
\end{align}
We assume that 
\begin{align}\label{12u2N}
\frac{1}{2}-\epsilon<u_{2N}<\frac{1}{2}.
\end{align}
From (\ref{unvne}) and (\ref{u2n}), we have 
\begin{align}\label{1-epsilon}
1-\epsilon<v_{2N}<1.
\end{align}
Let us consider the range of $u_{2N+1}$.
First, we assume that  $1-\epsilon<u_{2N+1}<1$.
Since $u_{2N+1}=\dfrac{1}{u_{2N}}-a_{2N}(\alpha)$,
we have 
\begin{align}\label{-1dfrac1}
-1+\dfrac{1}{u_{2N}}<a_{2N}(\alpha)<-\dfrac{7}{8}+\dfrac{1}{u_{2N}}.
\end{align}
From the inequalities (\ref{12u2N}) and (\ref{-1dfrac1}), 
we have
\begin{align}\label{1<a2N}
1<a_{2N}(\alpha)<\frac{43}{24}.
\end{align}
Since $v_{2N+1}=\frac{1}{v_{2N}}-a_{2N}(\alpha)$, from (\ref{1-epsilon}) and (\ref{1<a2N}), we have
$-\frac{19}{24}<v_{2N+1}<\frac{1}{7}$, which contradicts (\ref{u2n1}).
Therefore, we have
\begin{align}\label{-1<u_2N+1}
-1<u_{2N+1}<-1+\epsilon.
\end{align}
From (\ref{unvno}), (\ref{u2n1}), and (\ref{-1<u_2N+1}), we have
\begin{align}\label{-1<v_2N+1}
-2<v_{2N+1}<-2+\epsilon.
\end{align}
Let us consider the range of $u_{2N+2}$.
First, we assume that  $-\frac{1}{2}<u_{2N+2}<-\frac{1}{2}+\epsilon$.
Since $u_{2N+2}=\dfrac{1}{u_{2N+1}}-a_{2N+1}(\alpha)$,
we have 
\begin{align}\label{38dfrac1}
\dfrac{3}{8}+\dfrac{1}{u_{2N+1}}<a_{2N+1}(\alpha)<\dfrac{1}{2}+\dfrac{1}{u_{2N+1}}.
\end{align}
From the inequalities (\ref{-1<u_2N+1}) and (\ref{38dfrac1}),  
we have
\begin{align}\label{4356}
-\frac{43}{56}<a_{2N+1}(\alpha)<-\frac{1}{2}.
\end{align}
Since $v_{2N+2}=\frac{1}{v_{2N+1}}-a_{2N+1}(\alpha)$, from (\ref{-1<v_2N+1}) and (\ref{4356}), we have
$-\frac{1}{30}<v_{2N+2}<\frac{15}{56}$, which contradicts (\ref{u2n}).
Therefore, we have
\begin{align}\label{-1<u_2N+2}
\frac{1}{2}-\epsilon<u_{2N+2}<\frac{1}{2}.
\end{align}
Therefore, we obtain the following inequalities recursively for $n\geq N$:
\begin{align*}
&\frac{1}{2}-\epsilon<u_{2n}<\frac{1}{2},\ \  1-\epsilon<v_{2n}<1,\\
&\text{and}\\
&-1<u_{2n+1}<-1+\epsilon,\ \  -2<v_{2N+1}<-2+\epsilon. 
\end{align*}
By taking into account this fact and considering (\ref{limu2n}) and (\ref{limu2n+1}), we can derive claim (2)(i) of the lemma.
If we assume that $-\frac{1}{2}<u_{2N}<-\frac{1}{2}+\epsilon$, 
we can derive claim (2)(ii) of the lemma in a similar manner.
\end{proof}

\begin{lem}\label{Letalpha2adic}
Let $\mathbf{a}=(\alpha,\beta)\in (K/\mathbb{Q})^2$ with $\alpha\ne\beta$.
Then, there exists $n\in \mathbb{Z}_{\geq 0}$ such that 
$(\mathbf{a}_{(2n+1)})_{\langle 2 \rangle}\in D_{2}^1$.
\end{lem}
\begin{proof}
We put $(u_n,v_n):=(\mathbf{a}_{(n)})_{\langle 2 \rangle}$ for $n\in \mathbb{Z}_{> 0}$. 
According to Lemma \ref{lemord}, for all $n\geq 1$, the following inequalities hold:
\begin{align}\label{v_pu_2n}
v_p(u_{2n})\geq  0 \text{  and  }  v_p(u_{2n-1})> 0.
\end{align}

We assume that for all $n\in \mathbb{Z}_{\geq  0}$, $(u_{2n+1},v_{2n+1})\notin D_{2}^1$.
From Lemma \ref{Letalpha=}, we see that for all $n\in \mathbb{Z}_{> 0}$, $(u_{2n},v_{2n})\notin D_{2}^0$.
Therefore, we have for all $n\in \mathbb{Z}_{\geq 0}$
\begin{align}\label{v_pv_2n}
v_p(v_{2n+2})\geq 0 \text{  and  }  v_p(v_{2n+1})> 0.
\end{align}
For $n\in \mathbb{Z}_{> 0}$, we see that $\displaystyle u_{n+1}=\frac{1}{u_{n}}-a_n(\alpha)$ and $\displaystyle v_{n+1}=\frac{1}{v_{n}}-a_n(\alpha)$
, which implies
\begin{align*}
|u_n-v_n|_2=|u_nv_n|_2|u_{n+1}-v_{n+1}|_2.
\end{align*}
Therefore, for $n\in \mathbb{Z}_{> 0}$, we have
\begin{align}\label{u0v02}
|u_1-v_1|_2=|u_1v_1|_2\ldots |u_nv_n|_2|u_{n+1}-v_{n+1}|_2\leq |u_1v_1|_2\ldots |u_nv_n|_2.
\end{align}
From (\ref{v_pu_2n}), (\ref{v_pv_2n}), and (\ref{u0v02}), 
as $n \to \infty$, we can conclude that $|u_1 - v_1|_2$ tends to 0, which implies $u_1=v_1$.
This contradicts that  $\alpha\ne\beta$.
\end{proof}

We would like to expect that $\{\frac{p_n}{q_n}\}$ converges to $\alpha_{\infty}$, but for now, let's present a weaker form.

\begin{lem}\label{convergentsp2adicda}
For  all $n\in \mathbb{Z}_{> 0}$, it holds that $|\alpha_{\infty}-\frac{p_n}{q_n}|<1$.
\end{lem}

\begin{proof}
Let $n\in \mathbb{Z}_{> 0}$.
Let $\mathbf{a}=(\alpha,\frac{p_n}{q_n})$.
We apply Algorithm (Definition \ref{algol3}) to $\mathbf{a}$.
By induction, we obtain the following:\\
For $1\leq j \leq n$,
\begin{align*}
\mathbf{a}_{(j)}=(\alpha_{(j)},[0;a_j(\alpha),a_{j+1}(\alpha),\ldots,a_{n}(\alpha)]).
\end{align*}
and
\begin{align*}
\mathbf{a}_{(n+1)}=(\alpha_{(n+1)},0).
\end{align*}
From Lemma \ref{alphainK}, we see that $(\alpha_{(n+1)})_{\infty}\in I_{\infty}^{\tau(n+1)}$.
Therefore, we obtain $(\mathbf{a}_{(n+1)})_{\infty}\notin D_{\infty}^{\tau(n+1)}$. 
Hence, from Lemma \ref{Dinfty}, we have $(\mathbf{a}_{(n)})_{\infty}\notin D_{\infty}^{\tau(n)}$.
By induction, we obtain $(\mathbf{a}_{(1)})_{\infty}\notin D_{\infty}^{1}$.
Therefore, we have $|\alpha_{\infty}-\frac{p_n}{q_n}|<1$.
\end{proof}

The following corollary follows from the proof of Lemma \ref{convergentsp2adicda}.

\begin{cor}\label{convergentsp2adicdaco}
For  all $n,m\in \mathbb{Z}_{> 0}$ with $m\leq n$, it holds that 
\begin{align*}
|(\alpha_{(m)})_{\infty}-\left[0;a_m(\alpha),a_{m+1}(\alpha),\ldots,a_{n}(\alpha)\right]|<\dfrac{1+\tau(m)}{2}.
\end{align*}
\end{cor}

We observe that, employing the same reasoning as in Lemma \ref{convergentsp2adic}, the continued fraction $[a_{n-1}(\alpha);\ldots,a_{1}(\alpha)]$ is well-defined for $n\in \mathbb{Z}_{> 1}$, and this continued fraction is addressed in the following lemma.

\begin{lem}\label{inversecontinued}
For  all $n\in \mathbb{Z}_{> 1}$, it holds that
\begin{align*}
\left|(\alpha_{(n)})_{\infty}+[a_{n-1}(\alpha);\ldots,a_{1}(\alpha)]\right|
>\dfrac{1}{2}(1+\tau(n)).
\end{align*}
\end{lem}
\begin{proof}
Let $\mathbf{a}=(\alpha,\infty)$.
We apply Algorithm (Definition \ref{algol3}) to $\mathbf{a}$.
By induction, we obtain the following:\\
For  $n\in \mathbb{Z}_{\geq 2}$
\begin{align*}
\mathbf{a}_{(n)}=(\alpha_{(n)},-[a_{n-1}(\alpha);\ldots,a_{1}(\alpha)]).
\end{align*}
In fact, we have $\mathbf{a}_{(2)}=(\alpha_{(2)}, -a_{1}(\alpha))=(\alpha_{(2)}, -[a_{1}(\alpha)])$.
We assume that the claim holds for $n>1$.
Then, we have 
\begin{align*}
\mathbf{a}_{(n+1)}=&\left(\alpha_{(n+1)},-\dfrac{1}{[a_{n-1}(\alpha);\ldots,a_{1}(\alpha)]}-a_{n}(\alpha)\right)\\
=&(\alpha_{(n+1)},-[a_{n}(\alpha);\ldots,a_{1}(\alpha)]).
\end{align*}
Since $\mathbf{a}\in \mathcal{D}_{\infty}^{1}$ holds, 
based on Lemma \ref{Dinfty}, \ref{alphainK}, and Remark \ref{infinity}, we can conclude that for $n \in \mathbb{Z}_{\geq 1}$, $\mathbf{a}_{(n)}$ belongs to $\mathcal{D}_{\infty}^{\tau(n)}$.
This result implies the statement of the theorem.
\end{proof}

Following Lemma is proved easily, so that we omit the proof.
\begin{lem}\label{inversecontinued2}
For  $n\in \mathbb{Z}_{\geq 1}$, it holds that
\begin{align*}
\dfrac{q_{n-1}}{q_{n}}=[0;a_{n}(\alpha),\ldots,a_{1}(\alpha)].
\end{align*}

\end{lem}

\begin{lem}\label{boundf}
Let $\alpha\in K/\mathbb{Q}$ with property $I$.
For all $n\in \mathbb{Z}_{\geq 1}$, it holds that
\begin{align*}
\left|\alpha_{\infty}-\dfrac{p_n}{q_n}\right|<\dfrac{2}{(1+\tau(n))q^2_n}.
\end{align*}
\end{lem}

\begin{proof}
We have
\begin{align}
&\alpha-\dfrac{p_n}{q_n}=\dfrac{p_{n-1}\alpha_{(n+1)}+p_{n}}{q_{n-1}\alpha_{(n+1)}+q_{n}}-\dfrac{p_n}{q_n}=\dfrac{(-1)^n\alpha_{(n+1)}}{q_n(q_{n-1}\alpha_{(n+1)}+q_{n})}\nonumber\\
&=\dfrac{(-1)^n}{q^2_n\left(\dfrac{q_{n-1}}{q_n}+\dfrac{1}{\alpha_{(n+1)}}\right)}=\dfrac{(-1)^n}{q^2_n\left(\dfrac{q_{n-1}}{q_n}+a_{n+1}(\alpha)+\alpha_{(n+2)}\right)}\label{q2nleft}\\
&=\dfrac{(-1)^n}{q^2_n\left([a_{n+1}(\alpha);a_{n}(\alpha),\ldots,a_{1}(\alpha)]+\alpha_{(n+2)}\right)}.\nonumber
\end{align}
Therefore, from Lemma \ref{inversecontinued}, we have the claim of the lemma.
\end{proof}

\begin{lem}\label{boundf2}
Let $\alpha\in K/\mathbb{Q}$ with property $I$.
For  all $n\in \mathbb{Z}_{\geq 1}$, it holds that
\begin{align*}
v_2\left(\alpha_{\langle 2 \rangle}-\dfrac{p_n}{q_n}\right)=-2v_2(q_n)-v_2(a_{n+1}(\alpha)).
\end{align*}
\end{lem}

\begin{proof}
From the equation (\ref{q2nleft}), we have
\begin{align}\label{ord2leftal}
v_2\left(\alpha_{\langle 2 \rangle}-\dfrac{p_n}{q_n}\right)=-2v_2(q_n)-v_2\left(\dfrac{q_{n-1}}{q_n}+a_{n+1}(\alpha)+(\alpha_{(n+2)})_{\langle 2 \rangle}\right).
\end{align}
From Lemma \ref{convergentsp2adic}, 
\begin{align}\label{qnright}
v_2\left(\dfrac{q_{n-1}}{q_n}\right)=-v_2(a_{n}(\alpha)).
\end{align}
First, we assume that $n$ is even.
From Lemma \ref{lemord2},  we see that $v_2(a_{n}(\alpha))\leq 0$, 
$v_2(a_{n+1}(\alpha))<0$, and  $v_2((\alpha_{(n+2)})_{\langle 2 \rangle})\geq 0$, which implies
\begin{align*}
v_2\left(\frac{q_{n-1}}{q_n}+a_{n+1}(\alpha)+(\alpha_{(n+2)})_{\langle 2 \rangle}\right)
=v_2(a_{n+1}(\alpha)).
\end{align*}
Thus, we get the claim of the lemma in this case.
Next, we assume that $n$ is odd.
From Lemma \ref{lemord2},  we see that $v_2(a_{n}(\alpha))<0$, 
$v_2(a_{n+1}(\alpha))\leq 0$, and  $v_2((\alpha_{(n+2)})_{\langle 2 \rangle})>0$, which implies
\begin{align*}
v_2\left(\frac{q_{n-1}}{q_n}+a_{n+1}(\alpha)+(\alpha_{(n+2)})_{\langle 2 \rangle}\right)
=v_2(a_{n+1}(\alpha)).
\end{align*}
Thus, we get the claim of the lemma.
\end{proof}

From now on, let $K$ be a quadratic field
that can be embedded into both $\mathbb{R}$ and $\mathbb{Q}_p$.
For $\alpha\in K/\mathbb{Q}$, we denote its conjugate by $\bar{\alpha}$.
We note that $\bar{\alpha}\in K$.

\begin{lem}\label{boundpf}
Let $\alpha\in K/\mathbb{Q}$ with property $I$.
There exists $C>0$ such that for  all $n\in \mathbb{Z}_{\geq 1}$, it holds that
$|a_{n}(\alpha)|_{2}<C$.
\end{lem}
\begin{proof}
such that 
Let $ax^2+bx+c$ be a minimal polynomial of $\alpha$ over $\mathbb{Z}$.
We put $g(x,y)=ax^2+bxy+cy^2$.
Let $\mathbf{a}=(\alpha,\bar{\alpha})$.
From Lemma \ref{Letalpha2adic}, there exists $j\in \mathbb{Z}_{>0}$
such that $(\mathbf{a}_{(2j+1)})_{\langle 2 \rangle}\in D_{2}^1$.
For the sake of simplicity, we assume that 
$(\alpha_{\langle 2 \rangle},\bar{\alpha}_{\langle 2 \rangle})\in D_{2}^1$.
Let $n\in \mathbb{Z}_{>0}$.
Then, from Lemma \ref{convergentsp2adicda} and \ref{boundf}, we have
\begin{align}
0<|g(p_n,q_n)|&=|a||q_n^2|\left|\dfrac{p_n}{q_n}-\alpha_{\infty}\right|\left|\dfrac{p_n}{q_n}-\overline{\alpha}_{\infty}\right|\nonumber\\
&<2|a|\left(\left|\dfrac{p_n}{q_n}\right|+|\overline{\alpha}_{\infty}|\right)\nonumber\\
&<2|a|\left(1+|\alpha_{\infty}|+|\overline{\alpha}_{\infty}|\right)=C_1.\label{2left1+alpha}
\end{align}
We put $m_n=-v_2(q_n)$.
From Lemma \ref{convergentsp2adic} we see that $-v_2(q_n)\geq -v_2(p_n)$.
Therefore, we can conclude that $2^{2m_n}g(p_n,q_n)\in \mathbb{Z}$.
Hence, from inequality (\ref{2left1+alpha}), we can conclude:
\begin{align}\label{2m_ng(p_n,q_n)}
0\leq v_2(2^{2m_n}g(p_n,q_n))< 2m_n +\log_2(|C_1|+1). 
\end{align}
On the other hand, from Lemma \ref{boundf2}, we have
\begin{align}\label{2m_n-ord}
v_2(2^{m_n}(p_n-q_n\alpha_{\langle 2 \rangle}))=2m_n-v_2(a_{n+1}(\alpha)).
\end{align}  
Since $v_2(p_n)=-m_n-v_2(a_1(\alpha))>-m_n$ and $v_2(\overline{\alpha}_{\langle 2 \rangle})\leq 0$, 
we have 
\begin{align}\label{overlinealpha}
v_2(2^{m_n}(p_n-q_n\overline{\alpha}_{\langle 2 \rangle}))=v_2(\overline{\alpha}_{\langle 2 \rangle})=C_2.
\end{align}
Therefore, from (\ref{2m_n-ord}) and (\ref{overlinealpha}), we have
\begin{align}\label{alphaC_2}
&v_2(2^{2m_n}g(p_n,q_n))=v_2(2^{m_n}(p_n-q_n\alpha_{\langle 2 \rangle})2^{m_n}(p_n-q_n\overline{\alpha}_{\langle 2 \rangle}))\nonumber\\
&=2m_n-v_2(a_{n+1}(\alpha))+C_2.
\end{align}  
From (\ref{2m_ng(p_n,q_n)}) and (\ref{alphaC_2}), we have
\begin{align*}
-v_2(a_{n+1}(\alpha))<\log_2(|C_1|+1)-C_2,
\end{align*}
which prove the claim of the lemma.

\end{proof}

\begin{lem}\label{alphainKda}
Let $\alpha$ belong to $K/\mathbb{Q}$ with property $I$, and let $\beta$ be an element in $K/\mathbb{Q}$ distinct from $\alpha$.
Let $\mathbf{a}=(\alpha,\beta)$.
Then, there exists $n\in \mathbb{Z}_{>0}$ such that
$(\mathbf{a}_{(2n-1)})_{\infty}\in D_{\infty}^1$.
\end{lem}

\begin{proof}
We put $(u_n.v_n)=(\mathbf{a}_{(n)})_{\infty}$ for $n\in \mathbb{Z}_{>0}$.
We assume that for all $n\in \mathbb{Z}_{>0}$, $(u_{2n-1},v_{2n-1})\notin D_{\infty}^1$.
Then,  according to Lemma \ref{Letalpha}, either (2)(i) or (2)(ii) in the lemma holds.
We assume that (2)(i) holds: 
\begin{align}\label{lim_nto2}
\lim_{n\to \infty} (u_{2n},v_{2n})=\left(\frac{1}{2},1\right) \text{  and  }
\lim_{n\to \infty} (u_{2n-1},v_{2n-1})=(-1,-2).
\end{align}  

From the fact that $u_{2n+1}=\dfrac{1}{u_{2n}}-a_{2n}(\alpha)$ for $n\in \mathbb{Z}_{>0}$,
we have
\begin{align}\label{lim_nto}
\lim_{n\to \infty} a_{2n}(\alpha)=3.
\end{align}
Since $a_{2n}(\alpha)\in \mathbb{Z}[\frac{1}{p}]$ for $n\in \mathbb{Z}_{>0}$ and from Lemma \ref{boundpf} there
exists $C>0$ such that $|a_{n}(\alpha)|_{2}<C$, 
with equation (\ref{lim_nto}), we can conclude that 
there exists $N_1\in \mathbb{Z}_{>0}$  
such that for all $n\in \mathbb{Z}_{>0}$ with $n\geq N_1$ $a_{2n}(\alpha)=3$.
Similarly, we have
there exists $N_2\in \mathbb{Z}_{>0}$  
such that for all $n\in \mathbb{Z}_{>0}$ with $n\geq N_2$ $a_{2n-1}(\alpha)=-\frac{3}{2}$.
We set $N=\max\{N_1,N_2\}$.
Then, we have for $n\in \mathbb{Z}_{\geq 0}$
\begin{align*}
&\begin{pmatrix}0&1\\1&a_{2N-1}(\alpha)\end{pmatrix}\ldots \begin{pmatrix}0&1\\1&a_{2(N+n-1)}(\alpha)\end{pmatrix}=
\begin{pmatrix}1&3\\-3/2&-7/2\end{pmatrix}^{n}\\
&=\begin{pmatrix}2(-1/2)^{n}-(-2)^{n}&2(-1/2)^{n}-2(-2)^{n}\\-(-1/2)^{n}+(-2)^{n}&-(-1/2)^{n}+2(-2)^{n}\end{pmatrix}.
\end{align*}
Therefore, from Theorem \ref{conver}, we obtain that in $2$-adic topology 
\begin{align*}
(\alpha_{(2N-1)})_{\langle 2 \rangle}=\lim_{n\to \infty}\dfrac{2(-1/2)^{n}-2(-2)^{n}}{-(-1/2)^{n}+2(-2)^{n}}=-2,
\end{align*}
which contradicts $\alpha\in K/\mathbb{Q}$.
We have a contradiction in the case that (2)(ii) holds in a similar manner.
\end{proof}

\begin{rem}\label{evenperiod}
When discussing the continued fraction expansion by our algorithm, we assume that the length of its period is limited to even numbers. 
The reason is that there is a difference in the expressions used to calculate the next term between even and odd indices.
We give an example.
Let $\alpha=2\sqrt{17}-8$. Then, $\alpha_{\infty}=0.246\cdots$.
We assume that $\alpha_{\langle 2 \rangle} \equiv 2 \mod 8$.
Then, we have 
\begin{align*}
\alpha=\left[0;\dfrac{9}{2},-\dfrac{3}{2},-\dfrac{3}{2},\overline{5,-\dfrac{5}{2}}\right].
\end{align*}
\end{rem}
As can be seen from the above, it is evident that $\alpha_{(4)}=\alpha_{(6)}$,
 thus resulting in a period length of $2$.
However, we have $\alpha_{(2)}=\alpha_{(5)}$, which is not reflected in the continued fraction expansion.

\begin{thm}\label{Galois}
Let $\alpha\in K/\mathbb{Q}$ have property $I$ and assume that
there exist a positive odd integer $n_1$ such that 
$\alpha_{(1)}=\alpha_{(n_1)}$.
Then, $(\alpha,\bar{\alpha})_{\langle 2 \rangle}\in D_{2}^1$ and  $(\alpha,\bar{\alpha})_{\infty}\in D_{\infty}^1$
hold. 
\end{thm}
\begin{proof}
Let $\mathbf{a}=(\alpha,\bar{\alpha})$.
We apply Algorithm (Definition \ref{algol3}) to $\mathbf{a}$.
From Lemma \ref{Letalpha2adic} and \ref{alphainKda}, there exist odd positive integers $N_1, N_2$
such that $(\mathbf{a}_{(N_1)})_{\langle 2 \rangle}\in D_{2}^1$ and 
$(\mathbf{a}_{(N_2)})_{\infty}\in D_{\infty}^1$.
Using Lemma \ref{Letalpha=} and \ref{Dinfty}, we see that
for all odd  $n\in \mathbb{Z}_{>0}$ with $n\geq \max\{N_1,N_2\}$ 
$(\mathbf{a}_{(n)})_{\langle 2 \rangle}\in D_{2}^1$ and 
$(\mathbf{a}_{(n)})_{\infty}\in D_{\infty}^1$.
Therefore, the assumption that $\alpha_{(1)}=\alpha_{(n_1)}$ implies the claim of the theorem. 
\end{proof}

\begin{rem}\label{Galois2}
The converse of Theorem \ref{Galois} is generally not true.
We give an example.
We assume that $(\sqrt{17})_{\langle 2 \rangle} \equiv 1 \mod 8$.
Let $\alpha=-\frac{49}{64}+\frac{9\sqrt{17}}{64}$.
Then, $\alpha_{\infty}=-0.1858\cdots$.
$\bar{\alpha}_{\infty}=-1.3454\cdots$. 
We obtain that
$v_2(\alpha_{\langle 2 \rangle})=3$ and $v_2(\bar{\alpha}_{\langle 2 \rangle})=-5$.
Therefore, it holds that 
$(\alpha,\bar{\alpha})_{\langle 2 \rangle}\in D_{2}^1$ and  $(\alpha,\bar{\alpha})_{\infty}\in D_{\infty}^1$.
However, we have
\begin{align*}
\alpha=\left[0;-\dfrac{41}{8},\overline{-3,-\dfrac{3}{2},\dfrac{5}{2},\dfrac{37}{4},-3,\dfrac{5}{4},-5,-\dfrac{33}{8}}\right].
\end{align*}
\end{rem}

\begin{thm}\label{inftyp_nq_n}
Let $\alpha\in K/\mathbb{Q}$ have property $I$.
Then, 
\begin{align*}
\alpha_{\infty}=\lim_{n\to \infty}\dfrac{p_n}{q_n}.
\end{align*}
\end{thm}
\begin{proof}
We assume that $\alpha_{\infty}\ne \lim_{n\to \infty}\frac{p_n}{q_n}$.
For $n,m\in \mathbb{Z}_{>0}$ with $m\leq n$,
we denote $[0;a_m(\alpha),a_{m+1}(\alpha),\ldots,a_{n}(\alpha)]$ by $v_{m,n}$.
We define $v_{n+1,n}=0$.
Let $n\in \mathbb{Z}_{> 0}$.
Let $\mathbf{a}^{(n)}=(\alpha,\frac{p_n}{q_n})$.
From Corollary \ref{convergentsp2adicdaco}, for $1\leq m \leq n+1$, we have:
\begin{align*}
|(\alpha_{(m)})_{\infty}-v_{m,n}|\leq \dfrac{1}{2}(1+\tau(m)).
\end{align*}
We set that for $m\in \mathbb{Z}_{>0}$ $u_m=(\alpha_{(m)})_{\infty}$.
We obtain the following inequality in a similar manner as inequality (\ref{u0v0}):
for $n\in \mathbb{Z}_{>0}$, 
\begin{align}\label{u0v00}
|u_1-v_{1,n}|=|u_1v_{1,n}|\ldots |u_nv_{n,n}||u_{n+1}-v_{n+1,n}|\leq |u_1v_{1,n}|\ldots |u_nv_{n,n}|.
\end{align}
If we assume that
there exists $0<\epsilon<1$ such that for infinitely many $j\in \mathbb{Z}_{>0}$
$|u_{2j}|<\dfrac{1}{2}\epsilon$,  a similar argument to Lemma \ref{Letalpha} implies 
that $\lim_{n\to \infty} |u_1-v_{1,n}|=0$, which contradicts the assumption.
Therefore, we have
\begin{align}\label{limu2n2n}
\lim_{n\to \infty}|u_{2n}|=\frac{1}{2}.
\end{align}
Similarly, we have
\begin{align}\label{limu2n+12n+1}
\lim_{n\to \infty}|u_{2n+1}|=1.
\end{align}
We set
\begin{align*}
&A_{+}=\{n\in \mathbb{Z}_{>0}\mid n \equiv 0 \mod 2, u_{n}\geq 0\},\\
&A_{-}=\{n\in \mathbb{Z}_{>0}\mid n \equiv 0 \mod 2, u_{n}< 0\},\\
&B_{+}=\{n\in \mathbb{Z}_{>0}\mid n \equiv 1 \mod 2, u_{n}\geq 0\},\\
&B_{-}=\{n\in \mathbb{Z}_{>0}\mid n \equiv 1 \mod 2, u_{n}< 0\}.
\end{align*}
We assume that $\sharp A_{+}=\infty$.
We also assume that there exist infinitely many $j\in \mathbb{Z}_{>0}$
such that $2j\in A_{+}$ and $2j+1\in B_{+}$. 
Then, we have
\begin{align*}
\lim_{{\substack{j\to \infty \\ 2j\in A_+,2j+1\in B_+ }}}a_{2j}(\alpha)
=\lim_{{\substack{j\to \infty \\ 2j\in A_+,2j+1\in B_+ }}}\left(\dfrac{1}{u_{2j}}-u_{2j+1}\right)=1.
\end{align*}
Using the same method as the proof of Lemma \ref{alphainKda}, we  can conclude that
there exists an $N_1 \in \mathbb{Z}_{>0}$ such that for all $n \in \mathbb{Z}_{>0}$ 
with $n \geq N_1$, if $2n \in A_+$ and $2n+1 \in B_+$, then $a_{2n} = 1$.
Let  $k \in \mathbb{Z}_{>0}$ satisfy  $k \geq N_1$,  $2k \in A_+$, and $2k+1 \in B_+$.
Then, $u_{2k}$ and $u_{2k+1}$ satisfy following:
\begin{align*}
&u_{2k}\geq 0, \ \ u_{2k+1}\geq 0,\\
&|u_{2k}|\leq \dfrac{1}{2}, \ \ |u_{2k+1}|\leq 1,\\
&u_{2k+1}=\dfrac{1}{u_{2k}}-1,
\end{align*} 
which imply that $u_{2k}=\frac{1}{2}$ and $u_{2k+1}=1$.
This contradicts that $\alpha\in K/\mathbb{Q}$.
Therefore, there exists an $N_2 \in \mathbb{Z}_{>0}$ such that for all $n \in \mathbb{Z}_{>0}$ 
with $n \geq N_2$, if $2n \in A_+$, then $2n+1 \in B_-$.
Next, we  assume that there exist infinitely many $j\in \mathbb{Z}_{>0}$
such that $2j\in A_{+}$ and $2j+2\in A_{-}$. 
Then, using the same method as in the previous argument, we arrive at  a contradiction.
Therefore, there exists an $N_3 \in \mathbb{Z}_{>0}$ such that for all $n \in \mathbb{Z}_{>0}$ with
$n\geq N_3$, $2n \in A_+$ and $2n+1 \in B_-$.
Then, employing the same method as the proof of Lemma \ref{alphainKda},
we can deduce that there exists an $N_4 \in \mathbb{Z}_{>0}$ such that for all $n \in \mathbb{Z}_{>0}$ with
$n\geq N_4$,
$a_{2n}(\alpha)=3$ and
$a_{2n-1}(\alpha)=-\frac{3}{2}$,  leading to a similar contradiction.
For the case where $\sharp A_{-}=\infty$, we encounter a similar contradiction.
Thus, we obtain the claim of the lemma.
\end{proof}

The following corollary immediately follows from Theorem \ref{rational} and \ref{inftyp_nq_n}.

\begin{cor}\label{inftyp_nq_n2}
Let $\alpha\in K$ and $\{\frac{p_n}{q_n}\}$ be its convergents related to Algorithm (Definition \ref{algol1}).
Then, 
\begin{align*}
\alpha_{\infty}=\lim_{n\to \infty}\dfrac{p_n}{q_n}.
\end{align*}
\end{cor}

\begin{lem}\label{thereexistsC}
Let $\alpha\in K/\mathbb{Q}$ with property $I$.
Then, there exists $C>0$ such that for all $n \in \mathbb{Z}_{>0}$,
\begin{align*}
\left|\dfrac{q_{n-1}}{q_n}+a_{n+1}(\alpha)+(\alpha_{(n+2)})_{\infty}\right|<C.
\end{align*}
\end{lem}

\begin{proof}
Let $ax^2+bx+c$ be a minimal polynomial of $\alpha$ over $\mathbb{Z}$.
We put $g(x,y)=ax^2+bxy+cy^2$.
From the equation (\ref{q2nleft}), we have
\begin{align}
&|g(p_n,q_n)|=|a||q_n^2|\left|\dfrac{p_n}{q_n}-\alpha_{\infty}\right|\left|\dfrac{p_n}{q_n}-\overline{\alpha}_{\infty}\right|\nonumber\\
&=\dfrac{|a|}{\left|\dfrac{q_{n-1}}{q_n}+a_{n+1}(\alpha)+(\alpha_{(n+2)})_{\infty}\right|}\left|\dfrac{p_n}{q_n}-\overline{\alpha}_{\infty}\right|.
\label{|a||q_n^2|}
\end{align}
From Theorem \ref{inftyp_nq_n}, there exists $C_1>0$ such that for all $n \in \mathbb{Z}_{>0}$,
\begin{align}\label{p_nq_nover}
\left|\dfrac{p_n}{q_n}-\overline{\alpha}_{\infty}\right|<C_1.
\end{align}
On the other hand, from Lemma \ref{boundf2}, we have
\begin{align}\label{|g(p_n,q_n)|_2}
&|g(p_n,q_n)|_2=|a|_2|q_n^2|_2\left|\dfrac{p_n}{q_n}-\alpha_{_{\langle 2 \rangle}}\right|_2\left|\dfrac{p_n}{q_n}-\overline{\alpha}_{_{\langle 2 \rangle}}\right|_2
\nonumber\\
&=\dfrac{|a|_2}{|a_{n+1}(\alpha)|_2}\left|\dfrac{p_n}{q_n}-\overline{\alpha}_{_{\langle 2 \rangle}}\right|_2.
\end{align}
From Theorem \ref{conver}, there exists $C_2>0$ such that for all $n \in \mathbb{Z}_{>0}$,
\begin{align}\label{leftdfracp_n}
0<\left|\dfrac{p_n}{q_n}-\overline{\alpha}_{\langle 2 \rangle}\right|_2<C_2.
\end{align}
From (\ref{|g(p_n,q_n)|_2}), (\ref{leftdfracp_n}), we can conclude that
for all $n \in \mathbb{Z}_{>0}$,
\begin{align}\label{inequalityC_3}
0<|g(p_n,q_n)|_2<|a|_2C_2,
\end{align}
Considering $g(p_n,q_n)\in \mathbb{Z}[\frac{1}{2}]$, the inequality (\ref{inequalityC_3})
implies that for all $n \in \mathbb{Z}_{>0}$,
\begin{align}\label{inequalityC_32}
|g(p_n,q_n)|>\dfrac{1}{|a|_2C_2}.
\end{align}
From (\ref{|a||q_n^2|}), (\ref{p_nq_nover}), and (\ref{inequalityC_32}), we have for all $n \in \mathbb{Z}_{>0}$,
\begin{align*}
\left|\dfrac{q_{n-1}}{q_n}+a_{n+1}(\alpha)+(\alpha_{(n+2)})_{\infty}\right|<|a||a|_2C_1C_2.
\end{align*}
\end{proof}

\begin{thm}\label{ntoinfty|q_n|infty}
Let $\alpha\in K/\mathbb{Q}$ with property $I$.
Then, 
\begin{align*}
\lim_{n\to \infty}|q_n|=\infty.
\end{align*}
\end{thm}
\begin{proof}
From Theorem \ref{inftyp_nq_n}, Lemma \ref{thereexistsC}, and  the equation (\ref{q2nleft}), we have
\begin{align*}
\lim_{n\to \infty} |q^2_n|
=
\lim_{n\to \infty}\dfrac{|q^2_n|\left|\dfrac{q_{n-1}}{q_n}+a_{n+1}(\alpha)+(\alpha_{(n+2)})_{\infty}\right|}{\left|\dfrac{q_{n-1}}{q_n}+a_{n+1}(\alpha))+(\alpha_{(n+2)})_{\infty}\right|}=\infty.
\end{align*}
\end{proof}

\begin{cor}\label{ntoinfty|q_n|infty2}
Let $\alpha\in K/\mathbb{Q}$ and $\{\frac{p_n}{q_n}\}$ be its convergents related to Algorithm (Definition \ref{algol1}).

Then, 
\begin{align*}
\lim_{n\to \infty}|q_n|=\infty.
\end{align*}
\end{cor}

\begin{lem}\label{2adicalphainK}
Let $\alpha,\beta\in K/\mathbb{Q}$ with $(\alpha,\beta)_{\langle 2 \rangle}\in D_{2}^1$. 
Let $\mathbf{a}=(\alpha,\beta)$.
We denote $(u_n,v_n)=(\mathbf{a}_{(n)})_{\langle 2 \rangle}$ for $n\in \mathbb{Z}_{>0}$.
Then, $v_2(v_n)=v_2(a_{n-1}(\alpha))$ for $n\in \mathbb{Z}_{>1}$.
\end{lem}
\begin{proof}
For $n \in \mathbb{Z}_{>0}$, we observe that:
\begin{itemize}
\item If $n$ is odd, then $v_2(u_n) > 0$ and $v_2(v_n) \leq 0$.
\item If $n$ is even, then $v_2(u_n) \geq 0$ and $v_2(v_n) < 0$.
\end{itemize}
According to Lemma \ref{lemord2}, we have $v_2(u_n) = -v_2(a_{n}(\alpha))$ for $n \in \mathbb{Z}_{>0}$.
Thus, we find that $-v_2(v_n) > v_2(a_{n}(\alpha))$ for $n \in \mathbb{Z}_{>0}$.
Now let $n \in \mathbb{Z}_{>1}$. Since $v_n = \frac{1}{v_{n-1}} - a_{n-1}(\alpha)$, it follows that $v_2(v_n) = v_2(a_{n-1}(\alpha))$.
\end{proof}

Next, we have one of our main theorems.

\begin{thm}\label{Lagrange}
Let $\alpha\in K/\mathbb{Q}$ with property $I$.
Then, there exist $n_1, n_2\in \mathbb{Z}_{>0}$ with  $n_1<n_2$ such that
$\alpha_{(n_1)}=\alpha_{(n_2)}$ and $n_1\equiv n_2 \mod\ 2$.
\end{thm}
\begin{proof}
Let $\mathbf{a}=(\alpha,\bar{\alpha})$.
We apply Algorithm (Definition \ref{algol3}) to $\mathbf{a}$.
From Lemma \ref{Letalpha2adic} and \ref{alphainKda}, there exists $N_1, N_2\in \mathbb{Z}_{>0}$
such that $(\mathbf{a}_{(2N_1+1)})_{\langle 2 \rangle}\in D_{2}^1$ and 
$(\mathbf{a}_{(2N_2+1)})_{\infty}\in D_{\infty}^1$.
Using Lemma \ref{Letalpha=} and \ref{Dinfty}, we see that
$(\mathbf{a}_{(2N_3+1)})_{\langle 2 \rangle}\in D_{2}^1$ and 
$(\mathbf{a}_{(2N_3+1)})_{\infty}\in D_{\infty}^1$, where
$N_3=\max\{N_1,N_2\}$.
For the sake of simplicity, we assume that 
$(\mathbf{a})_{\langle 2 \rangle}\in D_{2}^1$ and 
$(\mathbf{a})_{\infty}\in D_{\infty}^1$.
We also remark that $\mathbf{a}_{n}=(\alpha_{(n)}, \overline{\alpha_{(n)}})$.
Let $p_1(x)=a_1x^2+b_1x+c_1$ be a minimal polynomial of $\alpha_{(1)}$ over $\mathbb{Z}$.
For $n\in \mathbb{Z}_{>1}$,we define $p_n(x)\in \mathbb{Z}[\frac{1}{2}][x]$ recursively as:
\begin{align}\label{p_n(x)}
p_n(x):=(x+a_{n-1}(\alpha))^2p_{n-1}\left(\dfrac{1}{x+a_{n-1}(\alpha)}\right).
\end{align}
From the fact that $\alpha_{(n)}=\frac{1}{\alpha_{(n-1)}}-a_{n-1}(\alpha)$,
we can conclude that  
$p_n(x)$ is a minimal polynomial of $\alpha_{(n)}$ over $\mathbb{Z}[\frac{1}{2}]$.
We set $a_nx^2+b_nx+c_n=p_n(x)$ for $n\in \mathbb{Z}_{>1}$.
For $n\in \mathbb{Z}_{>0}$, we define $d_n$ as the discriminant of $p_n(x)$, which is given by
$d_n=b_n^2-4a_nc_n$.
The relation (\ref{p_n(x)}) between $p_n(x)$ and $p_{n-1}(x)$ 
implies $d_n=d_{n-1}$ for $n\in \mathbb{Z}_{>1}$. 
Let $n\in \mathbb{Z}_{>1}$.
We set $d=d_1$.
From Lemma \ref{2adicalphainK},  
we have 
\begin{align*}
v_2((\alpha_{(n)}-\overline{\alpha_{(n)}})_{\langle 2 \rangle})=v_2(a_{n-1}(\alpha)).
\end{align*}
Therefore, from Lemma \ref{boundpf}, there exists $C_1>0$ such that
\begin{align}\label{C_1v_2}
-C_1<v_2((\alpha_{(n)}-\overline{\alpha_{(n)}})_{\langle 2 \rangle})\leq 0.
\end{align} 
Since it holds that $d=d_n=a_n^2(\alpha_{(n)}-\overline{\alpha_{(n)}}))^2$, from the inequality (\ref{C_1v_2}),
we have
\begin{align}\label{dfrac12v_2}
\dfrac{1}{2}v_2(d)  \leq v_2(a_n)<\dfrac{1}{2}(v_2(d)+C_1).
\end{align}
On the other hand, from the fact that $ (\mathbf{a}_{(n)})_{\infty}\in D_{\infty}^{\tau(n)}$, we have
\begin{align*}
|(\alpha_{(n)}-\overline{\alpha_{(n)}})_{\infty}|> \dfrac{1}{2},
\end{align*}
which implies 
\begin{align}\label{|a_n|<|d|}
|a_n|<2\sqrt{d}.
\end{align}
From the inequalities (\ref{dfrac12v_2}) and  (\ref{|a_n|<|d|}), there exists $C_2>0$ such that
\begin{align*}
\sharp\{a_m\mid m\in \mathbb{Z}_{>0}\}<C_2.
\end{align*}
Since $b_n=-a_n(\alpha_{(n)}+\overline{\alpha_{(n)}})_{\langle 2 \rangle}$, from the inequalities (\ref{C_1v_2}) and (\ref{dfrac12v_2}), we can conclude that
\begin{align}\label{-C_1+dfrac}
-C_1+\dfrac{1}{2}v_2(d)  \leq v_2(b_n)<\dfrac{1}{2}(v_2(d)+C_1).
\end{align}
Now, we have
\begin{align}\label{|alpha_(n)}
|(\alpha_{(n)}-\overline{\alpha_{(n)}})_{\infty}|=\dfrac{\sqrt{d}}{|a_n|}\leq \dfrac{\sqrt{d}}{C_3},
\end{align}
where $C_3=\min \{|a_m|\mid m\in \mathbb{Z}_{>0}\}$.
The inequality (\ref{|alpha_(n)}) leads to 
\begin{align}\label{1+dfrac|d|}
|(\overline{\alpha_{(n)}})_{\infty}| \leq  |(\alpha_{(n)})_{\infty}|+|(\alpha_{(n)})_{\infty}-(\overline{\alpha_{(n)}})_{\infty}|
\leq 1+\dfrac{\sqrt{d}}{C_3}.
\end{align}
Therefore, we have 
\begin{align}\label{|b_n|=|a_n|}
|b_n|=|a_n||(\alpha_{(n)})_{\infty}+(\overline{\alpha_{(n)}})_{\infty}|<2\sqrt{d}\left(2+\dfrac{\sqrt{d}}{C_3}\right).
\end{align}
From the inequalities (\ref{-C_1+dfrac}) and  (\ref{|b_n|=|a_n|}), there exists $C_4>0$ such that
\begin{align*}
\sharp\{b_m\mid m\in \mathbb{Z}_{>0}\}<C_4.
\end{align*}
Since $c_n=a_n(\alpha_{(n)})_{\langle 2 \rangle}(\overline{\alpha_{(n)}})_{\langle 2 \rangle}$, similarly, we have
\begin{align}\label{v_2(c_n)}
\dfrac{1}{2}v_2(d)-C_1 \leq v_2(c_n)<\dfrac{1}{2}(v_2(d)+3C_1).
\end{align}
From the inequality (\ref{|a_n|<|d|}) and (\ref{1+dfrac|d|}), we have
\begin{align}\label{c_n|=|a_n||}
|c_n|=|a_n||(\alpha_{(n)})_{\infty}||(\overline{\alpha_{(n)}})_{\infty}|<2\sqrt{d}\left(1+\dfrac{\sqrt{d}}{C_3}\right).
\end{align}
From the inequalities (\ref{v_2(c_n)}) and  (\ref{c_n|=|a_n||}), there exists $C_5>0$ such that
\begin{align*}
\sharp\{c_m\mid m\in \mathbb{Z}_{>0}\}<C_5.
\end{align*}
Therefore, since $\{p_n(x) \mid n \in \mathbb{Z}_{>0}\}$ is a finite set, the theorem's statement can be easily deduced.
\end{proof}

\begin{cor}\label{Lagrange2}
Let $\alpha\in K/\mathbb{Q}$.
Then, there exist $n_1, n_2\in \mathbb{Z}_{\geq 0}$ with  $n_1<n_2$ such that
$\alpha_{n_1}=\alpha_{n_2}$ and $n_1\equiv n_2 \mod\ 2$.
\end{cor}

We can represent convergents as a ratio of integers as follows.

\begin{defn}\label{convergents3}
For $n\in \mathbb{Z}_{> 0}$, $p'_n\in\mathbb{Z}$ and $q'_n\in\mathbb{Z}_{> 0}$ are defined by
\begin{align*}
p'_n=2^{-\sum_{s=1}^{n}v_2(a_{(s)}(\alpha))}p_n\ \ \text{and}
\ \ q'_n=2^{-\sum_{s=1}^{n}v_2(a_{(s)}(\alpha))}q_n.
\end{align*}
\end{defn}

We remark that for $n\in \mathbb{Z}_{> 0}$ $\dfrac{p'_n}{q'_n}=\dfrac{p_n}{q_n}$ and $(p'_n,q'_n)=1$.
The following theorems demonstrate the quality of the approximation.

\begin{thm}\label{qualityofapproximation}
Let $\alpha\in K/\mathbb{Q}$ with property $I$.
There exists a constant $C>0$ such that 
for all $n\in\mathbb{Z}_{>0}$,  
\begin{align*}
\dfrac{C}{|q'_n|^{2}}<\left|\alpha_{\infty}-\dfrac{p'_n}{q'_n}\right|
\left|\alpha_{\langle 2 \rangle}-\dfrac{p'_n}{q'_n}\right|_2
< \dfrac{2}{(1+\tau(n))|a_{n+1}(\alpha)|_2|q'_n|^2}.
\end{align*}
\end{thm}
\begin{proof}
Let $n\in\mathbb{Z}_{>0}$.
From Lemma \ref{boundf}, we have
\begin{align*}
&\left|\alpha_{\infty}-\dfrac{p_n}{q_n}\right|< \dfrac{2}{(1+\tau(n))|q_n|^2},
\end{align*}
which implies 
\begin{align}\label{qboundf}
\left|\alpha_{\infty}-\dfrac{p'_n}{q'_n}\right|< \dfrac{2\cdot2^{-2\sum_{s=1}^{n}v_2(a_{(s)}(\alpha)) }}{(1+\tau(n))|q'_n|^2}.
\end{align}
From this inequality and Lemma \ref{boundf2}, we obtain the following.
\begin{align*}
&\left|\alpha_{\infty}-\dfrac{p'_n}{q'_n}\right|\left|\alpha_{\langle 2 \rangle}-\dfrac{p'_n}{q'_n}\right|_2
<\dfrac{2}{(1+\tau(n))|a_{n+1}(\alpha)|_2|q'_n|^2}.
\end{align*}
Thus, we have the right hand side of the equality of the claim of the theorem.

From the equation (\ref{q2nleft}) and Lemma \ref{boundf2}, we have
\begin{align*}
&\left|\alpha_{\infty}-\dfrac{p'_n}{q'_n}\right|\left|\alpha_{\langle 2 \rangle}-\dfrac{p'_n}{q'_n}\right|_2=\dfrac{1}{|q'_n|^2|a_{n+1}(\alpha)|_2\left|\dfrac{q_{n-1}}{q_n}+a_{n+1}(\alpha)+(\alpha_{(n+2)})_{\infty}\right|}.
\end{align*}
Therefore, from Lemma \ref{boundpf} and \ref{thereexistsC},
we have the left hand side of the equality of the claim of the theorem. 
\end{proof}

\begin{cor}\label{qualityofapproximation2-1}
Let $\alpha\in K/\mathbb{Q}$ with property $I$.
Then, for all $n\in\mathbb{Z}_{>0}$,  
\begin{align*}
\left|\alpha_{\infty}-\dfrac{p'_n}{q'_n}\right|
\left|\alpha_{\langle 2 \rangle}-\dfrac{p'_n}{q'_n}\right|_2
< \dfrac{1}{|q'_n|^{2}}.
\end{align*}
\end{cor}

\begin{proof}
Since if $n$ is even, then $\tau(n)=0$ and $|a_{n+1}(\alpha)|_2\geq 2$ and  
if $n$ is odd, then $\tau(n)=1$ and $|a_{n+1}(\alpha)|_2\geq 1$, 
we have $\dfrac{2}{(1+\tau(n))|a_{n+1}(\alpha)|_2}\leq 1$.
Thus, from Theorem \ref{qualityofapproximation}, we have  the claim of the corollary.
\end{proof}

\begin{defn}\label{convergents2}
Let $\alpha\in K/\mathbb{Q}$ and $\{\frac{p_n}{q_n}\}$ be its convergents related to Algorithm (Definition \ref{algol1}).
For $n\in \mathbb{Z}_{> 0}$, $p'_n\in\mathbb{Z}$ and $q'_n\in\mathbb{Z}_{> 0}$ are defined by
\begin{align*}
p'_n=\begin{cases}
2^{-(\sum_{s=1}^{n}v_2(a_{(s)}(\alpha))}p_n& \text{if }
v_2(a_0(\alpha))>0,\\
2^{-(\sum_{s=0}^{n}v_2(a_{(s)}(\alpha))}p_n& \text{if }
v_2(a_0(\alpha))\leq 0,\\
\end{cases}\\
q'_n=\begin{cases}
2^{-(\sum_{s=1}^{n}v_2(a_{(s)}(\alpha))}q_n& \text{if }
v_2(a_0(\alpha))>0,\\
2^{-(\sum_{s=0}^{n}v_2(a_{(s)}(\alpha))}q_n& \text{if }
v_2(a_0(\alpha))\leq 0.
\end{cases}
\end{align*}
\end{defn} 

We remark that for $n\in \mathbb{Z}_{> 0}$ $\dfrac{p'_n}{q'_n}=\dfrac{p_n}{q_n}$ and $(p'_n,q'_n)=1$.
The following corollary follows from Theorem \ref{qualityofapproximation}. We will leave the proof to the reader.

\begin{cor}\label{qualityofapproximation2-2}
Let $\alpha\in K/\mathbb{Q}$ and $\{\frac{p_n}{q_n}\}$ be its convergents related to Algorithm (Definition \ref{algol1}).
There exist $C\in\mathbb{R}_{>0}$such that 
for all $n\in\mathbb{Z}_{\geq 0}$,  
\begin{align*}
\dfrac{C}{|q'_n|^{2}}<\left|\alpha_{\infty}-\dfrac{p'_n}{q'_n}\right|
\left|\alpha_{\langle 2 \rangle}-\dfrac{p'_n}{q'_n}\right|_2
< \dfrac{2\max\{1,(|b_{0}(\alpha)|_2)^2\}}{(1+\tau(n))|b_{n+1}(\alpha)|_2|q'_n|^2}.
\end{align*}
\end{cor}

The following theorem provides a simultaneous rational approximation to a number in $K$
 in both $\mathbb{R}$ and $\mathbb{Q}_2$, making it one of the main results.

\begin{thm}\label{qualityofapproximation3}
Let $\alpha\in K/\mathbb{Q}$ with property $I$.
There exist $\gamma, C, C' \in\mathbb{R}_{>0}$ with $0<\gamma<1$ 
such that for all $n\in \mathbb{Z}_{>0}$, 
\begin{align*}
\left|\alpha_{\infty}-\dfrac{p'_n}{q'_n}\right|\leq \dfrac{C}{|q'_n|^{2-2\gamma}} \text{ and }
\left|\alpha_{\langle 2 \rangle}-\dfrac{p'_n}{q'_n}\right|_2\leq \dfrac{C'}{|q'_n|^{2\gamma}}.
\end{align*}
\end{thm}
\begin{proof}
Let $l\in \mathbb{Z}_{>0}$ be the length of the period of $\{\alpha_{(n)}\}$.
We assume that 
for $n_1, n_2\in \mathbb{Z}_{>0}$ with  $n_1<n_2$ such that
$\alpha_{(n_1)}=\alpha_{(n_2)}$, $n_1\equiv n_2 \mod\ 2$, and
$n_2-n_1=l$.
Let $0\leq i\leq l-1$.
Let $n>n_2$ for $n\in \mathbb{Z}_{>0}$ with $n\equiv i \mod l$.
We take integers $j_1$ and $k$ such that $n_1 \leq j_1 < n_2$ and $n = j_1 + kl$.
We put 
\begin{align*}
M=\begin{pmatrix}0&1\\1&a_{j_1+1}(\alpha)\end{pmatrix}\begin{pmatrix}0&1\\1&a_{j_1+2}(\alpha)\end{pmatrix}\ldots \begin{pmatrix}0&1\\1&a_{j_1+l}(\alpha)\end{pmatrix}
.
\end{align*}
Since ${}^t(\alpha_{(j_1+1)},1)$ and ${}^t(\overline{\alpha_{(j_1+1)}},1)$ are eigenvectors of the matrix $M$, 
its eigenvalues $\lambda_1, \lambda_2$ belong to $K/\mathbb{Q}$.
For the sake of simplicity, we denote $(\lambda_i)_{\infty}$ by $\lambda_i$ for $i=1,2$.
Since $|\lambda_1\lambda_2|=1$, we can assume without loss of generality that $|\lambda_1|>1$ and $|\lambda_2|<1$.
We note that $\overline{\lambda_1}=\lambda_2$.
From the fact that
\begin{align*}
\begin{pmatrix}
p_{n-1}&p_{n}\\
q_{n-1}&q_{n}
\end{pmatrix}
=
\begin{pmatrix}
0&1\\
1&a_1(\alpha)
\end{pmatrix}
\cdots
\begin{pmatrix}
0&1\\
1&a_{j_1}(\alpha)
\end{pmatrix}
M^k,
\end{align*}
there exist $\delta_1, \delta_2\in K$ with $\delta_1\delta_2\ne 0$ such that
 $p_{n}=\delta_1\lambda^k_1+ \overline{\delta_1} \overline{\lambda_1}^k$ and 
$q_{n}=\delta_2\lambda^k_1+ \overline{\delta_2} \overline{\lambda_1}^k$.
Now, we put
\begin{align*}
&\xi_0=2^{-\sum_{s=1}^{j_1}v_2(a_{(s)}(\alpha))},\\
&\xi_1=2^{-\sum_{s=j_1+1}^{j_1+l}v_2(a_{(s)}(\alpha))},\\
&\gamma=\frac{\log \xi_1}{\log |\lambda_1\xi_1|}.
\end{align*}
We remark that $\lambda_1, \xi_1$ and $\gamma$ are constants independent of $n$. 
Since $|\lambda_1|>1$ and $\xi_1>1$, we see that $0<\gamma<1$.
From the inequality (\ref{qboundf}), we have 
\begin{align}\label{left|alphainfty}
\left|\alpha_{\infty}-\dfrac{p'_n}{q'_n}\right|< \dfrac{2(\xi_0\xi_1^k)^2}{(1+\tau(n))|q'_n|^2}
=\dfrac{2\xi_0^2|\lambda_1\xi_1|^{2k\gamma}}{(1+\tau(n))|q'_n|^2}.
\end{align}
Since
\begin{align*}
\lim_{k\to \infty}\dfrac{|q'_n|^{\gamma}}{|\lambda_1\xi_1|^{k\gamma}}
=\lim_{k\to \infty}\dfrac{|\xi_0\xi_1^k\delta_2\lambda^k_1+ \xi_0\xi_1^k\overline{\delta_2} \overline{\lambda_1}^k|^{\gamma}}{|\lambda_1\xi_1|^{k\gamma}}=|\xi_0\delta_2|^{\gamma},
\end{align*}
for sufficiently large $k$, we have
\begin{align}\label{|qngamma}
2|\xi_0\delta_2|^{\gamma}|\lambda_1\xi_1|^{k\gamma}>
|q'_n|^{\gamma}>\dfrac{|\xi_0\delta_2|^{\gamma}}{2}|\lambda_1\xi_1|^{k\gamma}.
\end{align} 
From the inequalities (\ref{left|alphainfty}) and (\ref{|qngamma}), 
for sufficiently large $k$, we have
\begin{align*}
\left|\alpha_{\infty}-\dfrac{p'_n}{q'_n}\right|< 
\dfrac{8\xi_0^2}{|\xi_0\delta_2|^{2\gamma}|q'_n|^{2-2\gamma}}.
\end{align*}
From Lemma \ref{boundf2} and the inequality (\ref{|qngamma}), for sufficiently large $k$, we have
\begin{align*}
\left|\alpha_{\langle 2 \rangle}-\dfrac{p'_n}{q'_n}\right|_2
\leq\dfrac{1}{\xi_0^2\xi_1^{2k}}=\dfrac{1}{\xi_0^2|\lambda_1\xi_1|^{2k\gamma}}
<\dfrac{4|\xi_0\delta_2|^{2\gamma}}{\xi_0^2|q'_n|^{2\gamma}}.
\end{align*}
Thus, we have the claim of the theorem.
\end{proof}

\begin{cor}\label{qualityofapproximation2-3}
Let $\alpha\in K/\mathbb{Q}$ and $\{\frac{p_n}{q_n}\}$ be its convergents related to Algorithm (Definition \ref{algol1}).
Let $l\in \mathbb{Z}_{>0}$ be the length of the period of $\{\alpha_{n}\}$.
There exist $\gamma, C, C' \in\mathbb{R}_{>0}$ with $0<\gamma<1$ 
such that for $n\in \mathbb{Z}_{>0}$,
\begin{align*}
\left|\alpha_{\infty}-\dfrac{p'_n}{q'_n}\right|\leq \dfrac{C}{|q'_n|^{2-2\gamma}} \text{ and }
\left|\alpha_{\langle 2 \rangle}-\dfrac{p'_n}{q'_n}\right|_2\leq \dfrac{C'}{|q'_n|^{2\gamma}}.
\end{align*}
\end{cor}

\section{Complex quadratic number}
Let $p$ be a prime number. Let $K$ be a quadratic field that has an embedding into $\mathbb{C}$ and $\mathbb{Q}_p$ respectively.
We assume that $\sigma_{\infty}$ gives an embedding into $\mathbb{C}$ and $\sigma_p$ gives an embedding into $\mathbb{Q}_p$.
We assume that $\sigma_{\infty}(K)\not\subset \mathbb{R}$, i.e., $\sigma_{\infty}(K)$ is a complex quadratic field.
Let $\alpha \in K$. We also denote  $\sigma_{\infty}(\alpha)$ by $\alpha_{\infty}$ and  $\sigma_{p}(\alpha)$ by $\alpha_{\langle p \rangle}$.
We denote the real part of $\alpha$ by $Re(\alpha)$ and the imaginary part by $Im(\alpha)$.
$F_{p,\epsilon}$ and $T_{p,\epsilon}$ for $\epsilon=0,1$ are applicable to $\alpha$, and therefore, Algorithms (Definition \ref{algol1}, Definition \ref{algol2}) are applicable to $\alpha$. We assume that various notations are the same as in the case when $K$ is a real quadratic field.
We will extend 'property $I$' slightly.
\begin{defn}\label{property Ic}
We say that $\alpha$ has property $I$, if  it satisfies $Re(\alpha)\in \left(-\frac{p}{2},\frac{p}{2}\right]$ and $v_p(\alpha_{\langle p \rangle})>0$.
\end{defn}

The following lemma can be proven similarly to Lemma \ref{boundf2}.

\begin{lem}\label{boundf3}
Let $\alpha\in K/\mathbb{Q}$ with property $I$.
For  all $n\in \mathbb{Z}_{\geq 1}$, it holds that
\begin{align*}
v_2\left(\alpha_{\langle 2 \rangle}-\dfrac{p_n}{q_n}\right)=-2v_2(q_n)-v_2(a_{n+1}(\alpha)).
\end{align*}
\end{lem}

Based on the numerical experiments, we conjecture that the sequence $\{\alpha_n\}$ will eventually become periodic
 (see section \ref{Numericalexperiments}).

\begin{thm}\label{qualityof complex}
Let $\alpha\in K/\mathbb{Q}$ with property $I$.
Let us assume that there exist $n_1, n_2\in \mathbb{Z}_{>0}$ with  $n_1<n_2$ such that
$\alpha_{(n_1)}=\alpha_{(n_2)}$ and $n_1\equiv n_2 \mod\ 2$.
Then, there exists $C\in\mathbb{R}_{>0}$ 
such that for all $n\in \mathbb{Z}_{>0}$,
\begin{align*}
\left|\alpha_{\langle 2 \rangle}-\dfrac{p'_n}{q'_n}\right|_2\leq \dfrac{C}{\max\{|p'_n|,|q'_n|\}^{2}}.
\end{align*}
\end{thm}
\begin{proof}
Let $l=n_2-n_1$.
Let $0\leq i\leq l-1$.
Let $n>n_2$ for $n\in \mathbb{Z}_{>0}$ with $n\equiv i \mod l$.
We take integers $j_1$ and $k$ such that $n_1 \leq j_1 < n_2$ and $n = j_1 + kl$.
We put 
\begin{align*}
M=\begin{pmatrix}0&1\\1&a_{j_1+1}(\alpha)\end{pmatrix}\begin{pmatrix}0&1\\1&a_{j_1+2}(\alpha)\end{pmatrix}\ldots \begin{pmatrix}0&1\\1&a_{j_1+l}(\alpha)\end{pmatrix}
\end{align*}
Since ${}^t(\alpha_{(j_1+1)},1)$ and ${}^t(\overline{\alpha_{(j_1+1)}},1)$ are eigenvectors of the matrix $M$, 
its eigenvalues $\lambda_1, \lambda_2$ belong to $K/\mathbb{Q}$.
Since $|\lambda_1\lambda_2|=1$ and $\lambda_1$ is a complex quadratic number, we have $|\lambda_1|=|\lambda_2|=1$.
We note that $\overline{\lambda_1}=\lambda_2$.
From the fact that
\begin{align*}
\begin{pmatrix}
p_{n-1}&p_{n}\\
q_{n-1}&q_{n}
\end{pmatrix}
=
\begin{pmatrix}
0&1\\
1&a_1(\alpha)
\end{pmatrix}
\cdots
\begin{pmatrix}
0&1\\
1&a_{j_1}(\alpha)
\end{pmatrix}
M^k,
\end{align*}
there exist $\delta_1, \delta_2\in K$ with $\delta_1\delta_2\ne 0$ such that
 $p_{n}=\delta_1\lambda^k_1+ \overline{\delta_1} \overline{\lambda_1}^k$ and 
$q_{n}=\delta_2\lambda^k_1+ \overline{\delta_2} \overline{\lambda_1}^k$.
Now, we put
\begin{align*}
&\xi_0=2^{-\sum_{s=1}^{j_1}v_2(a_{(s)}(\alpha))},\\
&\xi_1=2^{-\sum_{s=j_1+1}^{j_1+l}v_2(a_{(s)}(\alpha))}.
\end{align*}
From Lemma \ref{boundf3}, we have
\begin{align}\label{ineq11}
\left|\alpha_{\langle 2 \rangle}-\dfrac{p'_n}{q'_n}\right|_2<\dfrac{1}{(\xi_0\xi_1^k)^2}=\dfrac{|\delta_2\lambda^k_1+ \overline{\delta_2} \overline{\lambda_1}^k|^2}{|q'_n|^2}\leq \dfrac{|\delta_2|+|\overline{\delta_2}|}{|q'_n|^2}.
\end{align}
Similarly, we have
\begin{align}\label{ineq2}
\left|\alpha_{\langle 2 \rangle}-\dfrac{p'_n}{q'_n}\right|_2<\dfrac{1}{(\xi_0\xi_1^k)^2}=\dfrac{|\delta_1\lambda^k_1+ \overline{\delta_1} \overline{\lambda_1}^k|^2}{|p'_n|^2}\leq \dfrac{|\delta_1|+|\overline{\delta_1}|}{|p'_n|^2}.
\end{align}
Inequalities (\ref{ineq11}) and (\ref{ineq2}) lead to the theorem's conclusion.
\end{proof}

\begin{cor}\label{qualityof complex2}
Let $\alpha\in K/\mathbb{Q}$ and $\{\frac{p_n}{q_n}\}$ be its convergents related to Algorithm (Definition \ref{algol1}).
Let us assume that there exist $n_1, n_2\in \mathbb{Z}_{>0}$ with  $n_1<n_2$ such that
$\alpha_{n_1}=\alpha_{n_2}$ and $n_1\equiv n_2 \mod\ 2$.
There exists $C\in\mathbb{R}_{>0}$ 
such that for all $n\in \mathbb{Z}_{>0}$,
\begin{align*}
\left|\alpha_{\langle 2 \rangle}-\dfrac{p'_n}{q'_n}\right|_2\leq \dfrac{C}{\max\{|p'_n|,|q'_n|\}^{2}}.
\end{align*}
\end{cor}

\section{Numerical experiments and conjectures}\label{Numericalexperiments}
We demonstrate in Table \ref{t1} that for $1 < n \leq 200$, the continued fraction expansion of $\sqrt{n}$ obtained using Algorithm (Definition \ref{algol1}) satisfies the condition $\sqrt{n}=2^m\sqrt{k}$, where $m,k\in \mathbb{Z}{\geq 0}$, and $k$ is not a square of an integer. Additionally, $(\sqrt{k})_{\langle 2 \rangle}\in \mathbb{Q}_2$ and $(\sqrt{k})_{\langle 2 \rangle} \equiv 1 \mod 8$.
In this table, let $\gamma$ be the same as that in Theorem \ref{qualityofapproximation3},
 so that $\gamma_r = 2 - 2\gamma$ and $\gamma_2 = 2\gamma$.
Similarly, we demonstrate in Table \ref{t2} that for $1 < n \leq 200$, the continued fraction expansion of $(\sqrt{-n})_{\langle 2 \rangle}$ obtained using Algorithm (Definition \ref{algol1}) satisfies the condition $\sqrt{-n}=2^m\sqrt{-k}$, where $m,k\in \mathbb{Z}{\geq 0}$, and $k$ is not a square of an integer. Additionally, $(\sqrt{-k})_{\langle 2 \rangle}\in \mathbb{Q}_2$ and $(\sqrt{-k})_{\langle 2 \rangle} \equiv 1 \mod 8$.
As seen in Table \ref{t2}, for all $1 < n \leq 200$, $(\sqrt{-n})_{\langle 2 \rangle}\in \mathbb{Q}_2$ have eventually periodic expansions. Furthermore, we have confirmed that this holds true for all $1 < n \leq 10000$, which amounts to a total of 1665 cases.
We give a following conjecture.

\begin{conj}
Let $K$ be a quadratic field that has an embedding into $\mathbb{C}$ and $\mathbb{Q}_2$ respectively.
Let $\alpha\in K/\mathbb{Q}$ and  $\{\alpha_{n}\}$ be the sequence obtained by applying Algorithm (Definition \ref{algol1}) to $\alpha$. Then, $\{\alpha_{n}\}$ becomes eventually periodic.
\end{conj}

In Table \ref{t3} we show 
the total count when a period is detected in the continued fraction expansion of $\sqrt{n}$ with $1<n\leq 10000$ within $1000$ steps for $p<100$. 
Here, we have $n = p^{2m}k$, where $m$ is a non-negative integer ($m\in \mathbb{Z}_{\geq 0}$), and $k$ is a positive non-square integer that is also a quadratic residue below $\frac{p}{2}$ modulo $p$.
As shown in Table \ref{t3}, for all $3 \leq p \leq 23$, the continued fraction expansions become periodic in every case.
Even when $n<20000$, there are a few cases where the period length exceeds 1000. However, the expansions become periodic in all instances.

We give a following conjecture.

\begin{conj}
Let $p$ be a prime with $3 \leq p \leq 23$.
Let $K$ be a quadratic field that has an embedding into $\mathbb{R}$ and $\mathbb{Q}_p$ respectively.
Let $\alpha\in K/\mathbb{Q}$ and  $\{\alpha_{n}\}$ be the sequence obtained by applying Algorithm (Definition \ref{algol1}) to $\alpha$. Then, $\{\alpha_{n}\}$ becomes eventually periodic.
\end{conj}

\begin{table}[H]
\caption{Continued fraction expansion of $\sqrt{n}$ with $n>0$}
\label{t1}
\begin{tabular}{l|l|l|l}
\hline
$\sqrt{n}$& continued fraction expansion &$\gamma_r$&$\gamma_{2}$ \\
\hline
$\sqrt{17}$&$[5; -3/4, -3, \overline{5/2, -5}]$&1.54328
&0.45672
\\
$\sqrt{33}$&$[5;7/4,\overline{-3, 9/4}]$&1.04289
&0.957107
\\
$\sqrt{ 41 }$&$[
 7 ;
 -3/2 ,
\overline{ -5 ,
 -5/4}
]$&1.20362
&0.796382
\\
$\sqrt{ 57 }$&$[
 7 ;
 3/2 ,
\overline{ 5/2 ,
 5/4 }
]$&0.867752
&1.13225
\\
$\sqrt{ 65 }$&$[
 9 ;
 -5/8 ,
 -5/2 ,
\overline{ 9/2 ,
 -9/2 }
]$&1.35332
&0.646676
\\
$\sqrt{ 68 }$&$[
 8 ,
 9/2 ,
 -3/2 ,
 -3/2 ,
\overline{ 5 ,
 -5/2 }
]$&1.54328
&0.45672
\\
$\sqrt{ 73 }$&$[
 9 ;
 -7/4 ,
\overline{ -3 ,
 13/8 ,
 -3 ,
 -5/4}
]$&0.913949
&1.08605
\\
$\sqrt{ 89 }$&$[
 9 ;
 11/4 ,
 \overline{ -3 ,
 3/2 ,
 -13/2 ,
 3/2 ,
 -3 ,
 13/4} 
]$&1.17433
&0.825669
\\
$\sqrt{ 97 }$&$[
 9 ;
 13/8 ,
 -3/2 ,
 -5/4 ,
 \overline{-9 ,
 -9/4 }
]$&1.38203
&0.61797
\\
$\sqrt{ 105 }$&$[
 11 ;
 -3/2 ,
 \overline{13/2 ,
 -13/8 }
]$&0.869728
&1.13027
\\
$\sqrt{ 113 }$&$[
 11 ;
 -5/2 ,
 \overline{-9/2 ,
 -9/4}
]$&1.08953
&0.910471
\\
$\sqrt{ 129 }$&$[
 11 ;
 5/2 ,
 \overline{3 ,
 9/4 }
]$&1.21723
&0.78277
\\
$\sqrt{ 132 }$&$[
 12 ;
 -3/2 ,
 -3 ,
 \overline{3/2 ,
 -9/2}
]$&1.04289
&0.957107
\\
$\sqrt{ 137 }$&$[
 11 ;
 3/2 ,
 \overline{-13 ,
 13/8 }
]$&1.1728
&0.827195
\\
$\sqrt{ 145 }$&$[
 13 ;
 -5/4 ,
 5 ,
 -13/2 ,
 5 ,
 \overline{-3/4 ,
 -3 ,
 3/4 ,
 3 }
]$&0.821856
&1.17814
\\
$\sqrt{ 153 }$&$[
 13 ;
 -9/8 ,
 \overline{-5/4 ,
 -5/8 }
]$&0.396649
&1.60335
\\
$\sqrt{ 161 }$&$[
 13 ;
 -11/4 ,
 \overline{-3 ,
 3/2 ,
 -3 ,
 -9/4 }
]$&1.20791
&0.792091
\\
$\sqrt{ 164 }$&$[
 12 ;
 3/2 ,
 -7/2 ,
 \overline{-5/2 ,
 -5/2 }
]$&1.20362
&0.796382
\\
$\sqrt{ 177 }$&$[
 13 ;
 15/4 ,
 \overline{-3 ,
 25/16 ,
 -3 ,
 17/4 }
]$&0.883413
&1.11659
\\
$\sqrt{ 185 }$&$[
 13 ;
 17/8 ,
 \overline{-21/8 ,
 21/8 }
]$&0.541126
&1.45887
\\
$\sqrt{ 193 }$&$[
 13 ;
 5/4 ,
 \overline{-17/2 ,
 7/4 ,
 -3 ,
 3/2 ,
 -7/2 ,
 17/4 ,
 -7/2 ,
 3/2 ,
 -3 ,
 7/4 }
]$&1.03289
&0.967111\\
\hline
\end{tabular}
\end{table}

\newpage

\begin{table}[H]
\caption{Continued fraction expansion of $(\sqrt{n})_{\langle 2 \rangle}$ with $n<0$}
\label{t2}
\begin{tabular}{l|l}
\hline
$(\sqrt{n})_{\langle 2 \rangle}$& continued fraction expansion  \\
\hline
$(\sqrt{ -7 })_{\langle 2 \rangle}$&$[
  -1 ;
 1/2 ,
 -1 ,
 \overline{-1/2 ,
 1 }
]$\\
$(\sqrt{ -15 })_{\langle 2 \rangle}$&$[
 -1 ;
 1/2 ,
 -3/2 ,
 \overline{-1/2 ,
 1/2 }
]$\\
$(\sqrt{ -23 })_{\langle 2 \rangle}$&$[
 -1 ;
 1/2 ,
 -1 ,
 -1/4 ,
\overline{ -3/2 ,
 3/4 }
]$\\
$(\sqrt{ -28 })_{\langle 2 \rangle}$&$[
 0 ;
 -1/2 ,
 1 ,
 1/2 ,
 \overline{1 ,
 -1/2 }
]$\\
$(\sqrt{ -31 })_{\langle 2 \rangle}$&$[
 -1 ;
 1/2 ,
 -7/4 ,
 \overline{-1/2 ,
 1/4 }
]$\\
$(\sqrt{ -39 })_{\langle 2 \rangle}$&$[
  -1 ;
 1/2 ,
 -1 ,
 -1/2 ,
\overline{ -3/2 ,
 1/2 }
]$\\
$(\sqrt{ -47 })_{\langle 2 \rangle}$&$[
 -1 ;
 1/2 ,
 -5/2 ,
 3/4 ,
\overline{ 1 ,
 -1/2 ,
 1/2 ,
 -1/2 ,
 1 ,
 -1/4 }
]$\\
$(\sqrt{ -55 })_{\langle 2 \rangle}$&$[
 -1 ;
 1/2 ,
 -1 ,
 -5/8 ,
\overline{ -3/2 ,
 3/8 }
]$\\
$(\sqrt{ -60 })_{\langle 2 \rangle}$&$[
  0 ;
 -1/2 ,
 1 ,
 3/4 ,
\overline{ 1 ,
 -1/4 }
]$
\\
$(\sqrt{ -63 })_{\langle 2 \rangle}$&$[
 -1 ;
 1/2 ,
 -15/8 ,
 \overline{-1/2 ,
 1/8 }
]$\\
$(\sqrt{ -71 })_{\langle 2 \rangle}$&$[
  -1 ;
 1/2 ,
 -1 ,
 -1/2 ,
 -1 ,
 \overline{-1/2 ,
 5/4 ,
 -1/2 ,
 1 ,
 -5/8 ,
 1 }
]$\\
$(\sqrt{ -79 })_{\langle 2 \rangle}$&$[
  -1 ;
 1/2 ,
 -3/2 ,
 -5/4 ,
 1 ,
 -1/8 ,
\overline{ -7/4 ,
 7/8 }
]$\\
$(\sqrt{ -87 })_{\langle 2 \rangle}$&$[
 -1 ;
 1/2 ,
 -1 ,
 -3/4 ,
\overline{ -3/2 ,
 1/4 }
]$\\
$(\sqrt{ -92 })_{\langle 2 \rangle}$&$[
 0 ;
 -1/2 ,
 1 ,
 1/2 ,
 1/2 ,
\overline{ 3/4 ,
 -3/2 }
]$\\
$(\sqrt{ -95 })_{\langle 2 \rangle}$&$[
 -1 ;
 1/2 ,
 -5/4 ,
 -9/8 ,
 \overline{1 ,
 -1/2 ,
 1 ,
 -1/8 }
]$\\
$(\sqrt{ -103 })_{\langle 2 \rangle}$&$[
 -1 ;
 1/2 ,
 -1 ,
 -1/2 ,
 -3/4 ,
\overline{ -5/8 ,
 5/4 }
]$\\
$(\sqrt{ -111 })_{\langle 2 \rangle}$&$[
 -1 ;
 1/2 ,
 -5/2 ,
 3/2 ,
 \overline{-1 ,
 7/8 ,
 -1 ,
 1/2 }
]$\\
$(\sqrt{ -112 })_{\langle 2 \rangle}$&$[
 0 ;
 1/4 ,
 -3 ,
 \overline{-1/2 ,
 1 }
]$\\
$(\sqrt{ -119 })_{\langle 2 \rangle}$&$[
 -1 ;
 1/2 ,
 -1 ,
 -13/16 ,
 \overline{-3/2 ,
 3/16 }
]$\\
$(\sqrt{ -124 })_{\langle 2 \rangle}$&$[
0 ;
 -1/2 ,
 1 ,
 7/8 ,
 \overline{1 ,
 -1/8 }
]$\\
$(\sqrt{ -127 })_{\langle 2 \rangle}$&$[
 -1 ;
 1/2 ,
 -31/16 ,
 \overline{-1/2 ,
 1/16 }
]$
\\
$(\sqrt{ -135 })_{\langle 2 \rangle}$&$[
 -1 ;
 1/2 ,
 -1 ,
 -1/2 ,
 -1 ,
\overline{ -5/8 ,
 1 }
]$\\
$(\sqrt{ -143 })_{\langle 2 \rangle}$&$[
 -1 ;
 1/2 ,
 -3/2 ,
 -3/2 ,
 -1/2 ,
 7/8 ,
\overline{ -1/2 ,
 -1/2 ,
 1 ,
 -1/8 }
]$
\\
$(\sqrt{ -151 })_{\langle 2 \rangle}$&$[
 -1 ;
 1/2 ,
 -1 ,
 -5/4 ,
 5/4 ,
 \overline{1/2 ,
 -1 ,
 3/8 ,
 -1 ,
 1/2 ,
 -3/4 }
]$\\
$(\sqrt{ -156 })_{\langle 2 \rangle}$&$[
 0 ;
 -1/2 ,
 1 ,
 1/2 ,
 1 ,
\overline{ 3/4 ,
 -1 }
]$\\
$(\sqrt{ -159 })_{\langle 2 \rangle}$&$[
 -1 ;
 1/2 ,
 -11/4 ,
 3/2 ,
 -1 ,
\overline{ -1/8 ,
 1 ,
 -3/4 ,
 1 }
]$\\
$(\sqrt{ -167 })_{\langle 2 \rangle}$&$[
 -1 ;
 1/2 ,
 -1 ,
 -1/2 ,
 -5/2 ,
 \overline{1/4 ,
 5/4 ,
 1/4 ,
 -1/2 ,
 -5/8 ,
 -1/2 }
]$\\
$(\sqrt{ -175 })_{\langle 2 \rangle}$&$[
 -1 ;
 1/2 ,
 -5/2 ,
 15/8 ,
\overline{ -1 ,
 7/8 }
]$\\
$(\sqrt{ -183 })_{\langle 2 \rangle}$&$[
 -1 ;
 1/2 ,
 -1 ,
 -7/8 ,
\overline{ -3/2 ,
 1/8 }
]$\\
$(\sqrt{ -188 })_{\langle 2 \rangle}$&$[
  0 ;
 -1/2 ,
 1 ,
 5/4 ,
 -3/2 ,
 \overline{-1/2 ,
 1 ,
 -1/4 ,
 1 ,
 -1/2 ,
 1/2 }
]$\\
$(\sqrt{ -191 })_{\langle 2 \rangle}$&$[
 -1 ;
 1/2 ,
 -21/8 ,
 5/4 ,
\overline{ 1/4 ,
 1/4 ,
 -1 ,
 1/2 ,
 -1/2 ,
 -1/8 ,
 -1/2 ,
 1/2 ,
 -1 ,
 1/4 }
]$\\
$(\sqrt{ -199 })_{\langle 2 \rangle}$&$[
 -1 ;
 1/2 ,
 -1 ,
 -1/2 ,
 -1 ,
 \overline{-1/2 ,
 1 ,
 -1/2 ,
 5/4 ,
 -1/2 ,
 1 ,
 -1/2 ,
 1 ,
 -5/8 ,
 1 }
]$\\
\hline
\end{tabular}
\end{table}

\begin{table}[H]
    \centering
    \caption{Periodic Counts in Continued Fractions: $\sqrt{n}$ for $1<n<10000$ within $1000$ steps}
    \label{t3}
    \begin{tabular}{|c|c|c|c|}
        \hline
        $p$ & Total Count with Detected Period & Total Count with Undetected Period & Total Count \\
        \hline
        3 & 3652 & 0 & 3652 \\
        5 & 4067 & 0 & 4067 \\
        7 & 4278 & 0 & 4278 \\
        11 &4485&0&4485 \\
        13 &4543&0&4543 \\
        17 &4623&0&4623 \\
        19 &4652&0&4652 \\
        23 &4696&0&4696 \\
        29 &4443&291&4734 \\
        31 &4090&659&4749 \\
        37  &2964&1806&4770 \\
        41  &2676&2106&4782 \\
        43  &2455&2333&4788 \\
        47  &2219&2582&4801 \\
        53  &1837&2968&4805 \\
        59  &1807&3014&4821 \\
        61  &1753&3066&4819 \\
        67  &1669&3158&4827 \\
        71  &1612&3223&4835 \\
        73  &1562&3271&4833 \\
        79  &1399&3443&4842 \\
        83  &1340&3505&4845\\
        89  &1288&3557&4845 \\
        97  &1119&3733&4852 \\
        \hline
    \end{tabular}
\end{table}

\section*{Acknowledgements}
This research was supported by JSPS KAKENHI Grant Number JP22K12197.

\vspace{2cm}

\noindent
Shin-ichi Yasutomi: Faculty of Science, Toho University, 2-1 Miyama, Funabashi Chiba, 274-8510, JAPAN\\
{\it E-mail address: shinichi.yasutomi@sci.toho-u.ac.jp}
\end{document}